\documentclass[reqno]{amsart}
\usepackage{amssymb,amsmath,amsthm,amstext,amsfonts}
\usepackage{amsmath,amstext,amsthm,amsfonts}
\usepackage[dvips]{graphicx}
\usepackage[dvips]{graphicx}
\usepackage{color}

\usepackage[colorlinks=true,linkcolor=blue, urlcolor=red, citecolor=blue]{hyperref}	
\makeatletter \@addtoreset{equation}{section} \makeatother

\renewcommand\thetable{\thesection.\@arabic\c@table}

\theoremstyle{plain}
\newtheorem{maintheorem}{Theorem}

\newtheorem{maincorollary}{Corollary}

\newtheorem{mainproposition}{Proposition}

\newtheorem{theorem}{Theorem }[section]
\newtheorem{proposition}[theorem]{Proposition}
\newtheorem{lemma}[theorem]{Lemma}
\newtheorem{corollary}[theorem]{Corollary}

\theoremstyle{definition} \theoremstyle{remark}
\newtheorem{remark}[theorem]{Remark}

\newtheorem{definition}[theorem]{Definition}

\newcommand{\be} {\beta}        
    \newcommand{\Ga}{\Gamma}
\newcommand{\de} {\delta}       

\newcommand{\vep}{\varepsilon}

\newcommand{\dist}{\operatorname{dist}}

\newcommand{\topp}{\operatorname{top}}

\newcommand{\var}{\operatorname{var}}

\newcommand{\htop}{h_{\topp}}

\newcommand{\cU}{\mathcal{U}}

\newcommand{\cA}{\mathcal{A}}

\begin{document}

\title{On the rotation sets of generic homeomorphisms on the torus $\mathbb T^d$}

\author{H. Lima and P. Varandas}

\address{Heides Lima, Universidade Federal da Bahia\\
Av. Ademar de Barros s/n, 40170-110 Salvador, Brazil.}
\email{heideslima@gmail.com}

\address{Paulo Varandas, Departamento de Matem\'atica e Estat\'istica, Universidade Federal da Bahia\\
Av. Ademar de Barros s/n, 40170-110 Salvador, Brazil.}
\email{paulo.varandas@ufba.br, pcvarand@gmail.com}

\date{\today}

\begin{abstract}
We study the rotation sets for homeomorphisms homotopic to the identity on the torus $\mathbb T^d$, $d\ge 2$.
In the conservative setting, we prove that there exists a Baire residual subset of the set $\text{Homeo}_{0, \lambda}(\mathbb T^2)$ of conservative homeomorphisms homotopic to the identity so that the set of points with wild
pointwise rotation set is a Baire
residual subset in $\mathbb T^2$, and that it carries full topological pressure
and full metric mean dimension.
Moreover, we prove that for every $d\ge 2$ the rotation set of $C^0$-generic conservative homeomorphisms on $\mathbb T^d$ is convex.
Related results are obtained in the
case of dissipative homeomorphisms on tori.
The previous results rely on the description of the topological complexity of the set of points with wild historic behavior
and on the denseness of periodic measures for continuous maps with the gluing orbit property. 
\end{abstract}

\maketitle

\section{Introduction and statement of the main result}

In this paper we address and relate some fundamental concepts in topological dynamical systems, namely
topological pressure (including topological entropy), metric mean dimension and generalized rotation sets for
homeomorphisms on compact metric spaces.
Topological entropy and metric mean dimensions are two measurements of the dynamical complexity, which are particularly important
for continuous dynamical systems. While the first is a topological invariant, it is typically infinite for a $C^0$-Baire generic subset of
homeomorphisms on surfaces \cite{Yano}. On the other hand the second one, inspired by Gromov \cite{Gro} and proposed by
Lindenstrauss and Weiss, is a sort of dynamical analogue of the topological
dimension, depends on the metric and it is bounded above by the dimension of the ambient space \cite{LindW}.
In this way, the metric mean dimension may be used to distinguish the topological complexity of
surface homeomorphisms with infinite topological entropy.

Our main motivation is to describe rotation sets for homeomorphisms homotopic to the identity on tori.
The rotation number of a circle homeomorphism $f$, introduced by Poincar\'e \cite{Poincare}, is defined by
\begin{equation}\label{eq:rot-number}
\rho(f) = \lim_{n\to\infty} \frac{F^n(x)-x}n \; (\text{mod} 1)
\end{equation}
where $x\in \mathbb S^1$ and $F$ is a lift of the circle homeomorphism to $\mathbb R$. The rotation number
is independent of $F$ and $x$ and constitutes a very useful topological invariant (see e.g. \cite{cMvS}).
The situation changes drastically in the case of one-dimensional endomorphisms and higher-dimensional homeomorphisms.
This concept was first extended for continuous maps of degree one in the circle,  in which case the limit \eqref{eq:rot-number}
does not necessarily exist, its accumulation points form a (possibly degenerate) interval and such limit set  defines a rotation interval
which depends on the point $x$  (\cite{MPT}).
A generalization of rotation theory to a higher dimensional setting was studied by
Franks, Kucherenko, Kwapisz, Llibre, MacKay, Misiurewicz, Wolf and Ziemian among others (see \cite{1989F,FM,16KW, KWAP, MZ,MZ2} and references therein) for homeomorphisms homotopic to the identity, where the notion of rotation sets extend the concept of rotation number for  circle homeomorphisms. Although rotation sets are not a complete invariant,
their shapes can be used to describe properties of the dynamical system, as we now illustrate.
If $f$ is a homeomorphism on the torus $\mathbb T^d$ ($d\ge 2$) homotopic to the identity,
$\pi:\mathbb R^d\to\mathbb T^d=\mathbb R^d/\mathbb Z^d$
is the natural projection and $F: \mathbb R^d \to \mathbb R^d$
is a lift for $f$, the \emph{rotation set} of $F$ is defined by
\begin{equation}\label{eq:rot-set}
\rho(F) = \Big\{v\in \mathbb R^d\;:\; \text{exist }  z_i\in \mathbb R^2, \; n_i\in\mathbb N\; \text{so that}\;
	\lim_{i\to\infty} \frac{F^{n_i}(z_i)-z_i}{n_i} =v \Big\}.
\end{equation}
and the\emph{(pointwise) rotation set } $\rho(F,z)$ of a point $z\in \mathbb R^d$ is the set of the following accumulation vectors
\begin{eqnarray}\label{eq:rot-vect}
\rho(F, z) := \text{acc} \Big( \frac{F^{n}(z)-z}{n} \Big)_{n\ge 1}.
\end{eqnarray}
Given $x\in \mathbb T^d$ we define $\rho(F, x)$ by \eqref{eq:rot-vect} (note that the previous expression does
not vary in $\pi^{-1}(x)$).
The \emph{pointwise rotation set} of $F$ is $\rho_p(F) = \bigcup_{x\in \mathbb T^2} \rho(F,x)$.
The previous sets are compact and connected subsets of $\mathbb R^d$, and we will call them \emph{trivial} if they are reduced
to a single vector (see e.g.  Subsection~\ref{subsec:rotation-sets} and \cite{LM, MZ} for more details).
In the $2$-torus, each rotation set is convex (it may fail to be convex in higher dimensional torus)
but there are compact convex sets of the plane that are not the rotation set
of any torus homeomorphisms \cite{KWAP2}. 
Nevertheless, for every rational convex polygon $K\subset \mathbb R^2$ there exists a homeomorphism $f$ on $\mathbb T^2$
homotopic to the identity so that $\rho(F)=K$ \cite{KWAP}.

We will focus on the realization of convex sets as
rotation sets (see Subsection~\ref{subsec:rotation-sets}
for the definition).
More precisely, if $f$ is a homeomorphism  on $\mathbb T^d$, $g\ge 2$, and the map $F:\mathbb R^d \to \mathbb R^d$ is a lift:
\begin{enumerate}
\item given a compact and convex set $K\subset \rho(F)$ does there exist $x\in \mathbb T^d$ and $z\in \pi^{-1}(x)\in \mathbb R^d$ such that $\rho(F,z)=K$?
\item if the previous holds, what is the size of such set of points in $\mathbb T^d$?
\item how commonly (in $f$) is $\rho(F)$ convex?
\end{enumerate}
Concerning the first question we note that if $f$ is a homeomorphism isotopic to the identity on $\mathbb T^2$ and $F$ is a lift then:
(i) for every rational vector $v\in \rho(F)$ in the interior of $\rho(F)$ there exists a periodic point $x\in \mathbb T^d$
	so that $\rho(F,x)=v$ \cite{1989F};
(ii) for any vector $v$ in the interior of $\rho(F)$ there exists a non-empty compact set $\Lambda_v\subset \mathbb T^2$ so that $\rho(F,x)=v$ for every $x\in \Lambda_v$ and, under some mild assumptions, $f\mid_{\Lambda_v}$ has positive topological entropy \cite{MZ,AZ},
(iii) for any any compact connected $C$ is in the interior of the convex hull of vectors in $\rho(F)$ which represent periodic orbits of $f$  there exists a point $x\in \mathbb T^2$ so that $\rho(F, \tilde{x})=C$ \cite{LM}.

It seems that much less is known as an answer to the second question. Building over \cite{GK,P14} we prove that $C^0$-generic conservative homeomorphisms homotopic to the identity on $\mathbb T^2$ are so that the set of points for which the rotation vector is not well defined (equivalently, the limit defined by \eqref{eq:rot-vect} does not exist) form a Baire residual,
full topological pressure and full metric mean dimension subset of $\mathbb T^2$.
In the case of dissipative homeomorphisms homotopic to the identity we prove that the gluing orbit property is typical among the
isolated chain recurrent classes in the non-wandering set.  We use this fact to prove that for most surface homeomorphisms of 
$\mathbb T^2$ homotopic to the identity having rotation set with non-empty interior, the set of points with non-trivial pointwise rotation set is either empty or topologically large (Baire residual, full topological entropy and metric mean dimension) in a isolated chain recurrent class.

Finally, concerning the third question, we refer that Passeggi \cite{P14} proved that an open and dense subset set of
homeomorphisms on $\mathbb T^2$ homotopic to the identity so that the rotation set is a rational polygon.
Here we prove that $C^0$-generic conservative homeomorphisms homotopic to the identity on the torus $\mathbb T^d$, $d\ge 2$,
have a convex rotation set, providing an answer to this question. We also obtain related results in the dissipative context.

The previous results fit in a more general framework, namely the description of the topological complexity of the set of points
with historic behavior (also known as irregular, exceptional or non-typical points) from the topological viewpoint, and the density of periodic measures.
Given a continuous map  $f : X \rightarrow X$ on a compact metric space $(X, d)$ and a continuous observable
$\varphi : X \rightarrow \mathbb R^d$ ($d\ge 1$), the set of points with historic behavior with respect to $\varphi$ is
\begin{eqnarray*}
X_{\varphi, f} := \Big\{x \in X : \lim_{n\to\infty} \frac{1}{n}\sum_{i=0}^{n - 1}\varphi(f^i (x)) ~~\text{does not exist}  \Big\}.
\end{eqnarray*}
The term historic behavior was coined after some dynamics where the phenomena of the persistence of points with this
kind of behavior occurs \cite{Ruelle, Takens}.
Birkhoff's ergodic theorem (applied to the coordinates of $\varphi$)
ensures that $X_{\varphi_f}$ is negligible from the measure theoretic viewpoint, as it has zero measure
with respect to any invariant probability measure. It was first proved by Pesin and Pitskel, and by Barreira and Schmelling,
that in the case of subshifts of finite type,
conformal repellers and conformal horseshoes the sets $X_{\varphi,f}$ are either empty or carry full topological entropy,
and full Hausdorff dimension \cite{Barreira,PP}.
Several extensions of these results have been considered later on, building mainly over the concept of specification introduced
by Bowen in the early seventies and the concept of shadowing (see e.g. \cite{BarreiraS,CKL, DOT, LW,OlsenWinter,VT,TMP} and references therein).

Here we obtain yet another mechanism to describe the topological complexity of the set of points
with historic behavior, and to pave the way to multifractal analysis. In order to do so, we introduce the notion of relative
metric mean dimension. Then, given a continuous map with the gluing orbit property (a concept introduced in \cite{BV} in the context of topological
dynamical systems which bridges between uniform and
non-uniform hyperbolicity and extends the concept of specification) we prove that any non-empty set of points with
historic behavior has three levels of topological complexity: it is Baire generic, it has full topological pressure
and it has full metric mean dimension (Theorems~\ref{thm:A} and \ref{thm:B}). Moreover, we prove that the latter holds for typical pairs
$(f,\varphi)$ of homeomorphisms and continuous observables (Corollary~\ref{cor:A}), building over the fact, of independent
interest, that the gluing orbit property holds on isolated chain recurrent classes of $C^0$-generic homeomorphisms (Corollary~\ref{lem1}).

This paper is organized as follows. In Section~\ref{sec:state} we describe the setting,  state our main results and provide a discussion
on the arguments in the proofs. Some preliminaries on the topological invariants and notions of complexity are given in
Section~\ref{sec:prelim}. Section~\ref{sec:irregular} is devoted to the proof of the results on the set of points with wild historic
behavior for maps with the gluing orbit property. The results on the rotation sets for homeomorphisms homotopic to identity
are given in Sections~\ref{sec:homotopic1} and ~\ref{sec:homotopic2}. Finally, in Section ~\ref{sec:questions} we make some comments
and discuss possible directions of research.

\section{Statement of the main results}\label{sec:state}

\subsection{Pointwise rotation sets of homeomorphisms on the torus $\mathbb T^2$}

In this section we address the questions concerning the pointwise rotation sets of torus homeomorphisms homotopic to the identity.
We note that the pointwise rotation set may fail to be connected and all (see e.g. \cite[Example 1]{LM}).
Our first results ensure that this is not the typical situation 
in the case of volume preserving homeomorphisms.

\begin{maintheorem}\label{thm:C}
There exists a Baire residual subset $\mathfrak R_1 \subset \text{Homeo}_{0, \lambda}(\mathbb T^2)$ so that, for every $f\in \mathfrak{R}_1$ and every lift $F: \mathbb R^2 \to \mathbb R^2 $ of $f$:
\begin{enumerate}
\item the pointwise rotation set $\rho_p(F)$ is connected;
\item the set of points $x\in \mathbb T^2$ such that $\rho(F,{x})$ is non-trivial 
and coincides with $\rho_p(F)$
is a Baire residual subset of $\mathbb T^2$, it carries full topological pressure and full metric mean dimension in $\mathbb T^2$.
\end{enumerate}
\end{maintheorem}

Now we describe the counterpart of Theorem~\ref{thm:C} on the space $\text{Homeo}_{0}^{}(\mathbb T^2)$
of homeomorphisms homotopic to the identity.
Consider the set
$$
\mathcal A = \left\{f \in \text{Homeo}_{0}^{}(\mathbb T^2) : \text{int}\,\rho(F)\neq\emptyset\right\}.
$$ 
It is clear that the set $\mathcal A$  does not depend on the lift $F$ of $f$.
All homeomorphisms in $\cA$
have positive topological entropy \cite{LM}.
Let $CR(f)$ denote the chain recurrent set of $f$ (cf. Subsection~\ref{gop} for definitions). 
We prove the following:

\begin{maintheorem}\label{thm:D}
There exists a Baire residual subset $\mathfrak R_2 \subset \mathcal A$  so that, for every $f\in \mathfrak{R}_2$
there exists a positive entropy chain recurrent class $\Gamma\subset \Omega (f)$ such that $int(\rho(F\mid_{\pi^{-1}(\Gamma)})) \neq \emptyset$.  Moreover, if in addition $\Gamma$ is a isolated chain recurrent class then 
the set of points $x\in \Gamma$ for which $\rho(F,{x})$ is non-trivial  
is a Baire residual subset of $\Gamma$ that carries full topological entropy and full metric mean dimension in $\Gamma$.
\end{maintheorem}

\begin{remark}
Since $C^0$-generic homeomorphisms have infinite topological entropy \cite{Yano}, 
for every $\beta>0$ there exists a chain recurrent class $C_\beta \subset CR(f)$ such that $\htop(f\mid_{C_\be})\ge \be$
(by the variational principle it is enough to take chain recurrent classes containing supports of ergodic measures with 
arbitrarily large entropy). However, a priori the rotation set restricted to each of these chain recurrent classes 
(obtained in \cite{Yano} by the creation of pseudo-horseshoes in local 
perturbations of the dynamics) could have empty interior.   
\end{remark}

A construction of a smooth minimal diffeomorphism on the two-torus, homotopic to the identity, whose rotation set is a non-trivial
line segment and so that the pointwise rotation set is non-trivial for Lebesgue almost every point
has been recently announced in \cite{ALX}.
We also note that the proof of Theorem \ref{thm:C} uses that generic conservative homeomorphisms satisfy the specification property,
while in the dissipative setting, the specification property seldom occurs. For that reason a key ingredient used in the proof of 
Theorem \ref{thm:D} is that generic homeomorphisms restricted to isolated chain recurrent classes satisfy the gluing orbit property.  

\subsection{On the rotation set of homeomorphisms on the torus $\mathbb T^d$ ($d\ge 2$)} 

The shape of the different rotation sets
for an homeomorphism $f$ homotopic to identity on the torus $\mathbb T^d$
have drawn the attention since these have been introduced (see Subsection~\ref{subsec:rotation-sets} for definitions).
Focusing first on connectedness, the rotation set $\rho(F)$ (and each pointwise rotation set $\rho(F,x)$) is a compact
and connected set in $\mathbb R^d$ \cite{LM,MZ2}. However, the pointwise rotation set $\rho_p(F)$ may fail to be connected even when $d=2$ \cite{LM}. As for convexity,  $\rho(F)$ is convex when $d=2$, but there are higher dimensional examples where it fails to be convex \cite{MZ2}.

Our next result ensures that rotation sets of torus homeomorphisms are typically convex
(we refer the reader to Subsection~\ref{gop} for the notion of chain recurrence).

\begin{maintheorem}\label{thm:E}
For every $d\ge 2$:
\begin{enumerate}
\item there exists a Baire residual subset $\mathfrak R_3\subset \text{Homeo}_{0, \lambda}(\mathbb T^d)$ so that
$\rho(F)$ is convex, for every lift $F$ of a homeomorphism $f\in \mathfrak{R}_3$; and
\item 
there exists a Baire residual subset $\mathfrak R_4\subset \text{Homeo}_{0}(\mathbb T^d)$ so that
$\rho(F\mid_{\pi^{-1}(\Gamma)})$ is convex, for every isolated chain recurrent class $\Gamma\subset \Omega(f)$ and every lift $F$
of $f\in \mathfrak R_4$.
\end{enumerate}
\end{maintheorem}

While the rotation set is always connected, in the case of dissipative homeomorphisms $\text{Homeo}_0(\mathbb T^d)$ (e.g. Morse-Smale diffeomorphisms on the torus) the pointwise rotation set need not to be connected. If the pointwise rotation set is
connected then one can hope that the ``local" convexity  statement in item (2) can be used to prove the convexity of the rotation set.

\subsection{Points with historic behavior for maps with
gluing orbit property}

The results in this section, despite their own interest, will be key technical ingredients in the characterization of rotation sets for
homeomorphisms on tori. These applications motivate to describe the set of points with historic behavior for observables taking values on $\mathbb R^d$, $d\ge 1$,
and dynamical systems with the gluing orbit property (see Subsection~\ref{gop} for the definition).

Let $X$ denote a compact metric space, $f: X\to X$ be a continuous map, $d\ge 1$ be an integer and
$\varphi: X \to \mathbb R^d$ be a continuous observable.
Given $x\in X$, let us denote by $\mathcal V_\varphi(x)$
the (connected) set obtained as accumulation points of $(\frac1n\sum_{j=0}^{n-1} \varphi ( f^j(x)))_{n\ge 1}$.
In the higher dimensional setting context, 
$(d>1)$ the set $\mathcal V_\varphi =\bigcup_{x\in X} \mathcal V_\varphi(x) \subset \mathbb R^d$
of all vectors obtained as pointwise limits of Birkhoff averages
need not be connected or convex.

A point $x\in X$ has \emph{historic behavior} for $\varphi$
(also known as \emph{exceptional}, \emph{irregular} or \emph{non-typical} behavior) if the limit $\lim_{n\to\infty} \frac1n\sum_{j=0}^{n-1} \varphi ( f^j(x))$ does not exist.
Moreover, in the case that $\mathcal V_\varphi$ does not reduce to a single vector, we say that $x\in X$ has \emph{wild historic behavior} if $\mathcal V_\varphi(x)=\mathcal V_\varphi$.
In rough terms, a point has wild historic behavior if the Birkhoff averages have the largest oscillation in $\mathcal V_\varphi$.
We say that $B\subset X$ is \emph{Baire residual} if it contains a countable intersection of open and dense subsets of $X$.

Our first result asserts that, under a mild assumption,  if non-empty, the set of points with wild historic behavior is large from the category
point of view.

\begin{maintheorem}\label{thm:A}
Let $X$ be a compact metric space, let $f: X \to X$ be a continuous map with the gluing orbit property and let
$\varphi: X \to \mathbb R^d$ be continuous. Then:
\begin{enumerate}
\item either there is $v\in \mathbb R^d$ so that $\lim_{n\to\infty} \frac1n\sum_{j=0}^{n-1} \varphi( f^j(x))=v$ for all $x\in X$,
\item or the set $X_{\varphi,f}$ of points $x\in X$ so that the sequence $(\frac1n\sum_{j=0}^{n-1} \varphi( f^j(x)))_{n\ge 1}$
accumulates in a non-trivial connected 
subset of $\mathbb R^d$ is Baire residual on $X$.
\end{enumerate}
 Moreover, if $X_{\varphi,f} \neq \emptyset$ then $\mathcal V_\varphi$ is connected and the set of points with wild historic behavior is a Baire residual
subset of $X$.
\end{maintheorem}

The next result establishes that the set of points with historic behavior has also large complexity, now measured in terms of
topological entropy and metric mean dimension.  We refer the reader to
Subsection~\ref{sec:entropy-mmd} for the notions of full topological pressure and full metric mean dimension.

\begin{maintheorem}\label{thm:B} Let $f : X \rightarrow X$ be a continuous map with the gluing orbit property on compact metric space $X$ and let $\varphi : X \rightarrow \mathbb R^d$ be a continuous observable. Assume that $X_{\varphi, f}^{}\neq \emptyset$. Then $X_{\varphi, f}^{}$ carries full topological pressure and full metric mean dimension.
\end{maintheorem}

Under the previous assumptions, the set of points with historic behavior for $\varphi$ is empty if and only if there exists $v\in \mathbb R^d$ so that  $\int \varphi\, d\mu=v$ for every $f$-invariant probability measure (cf. Lemma~\ref{obsat}). 
The second property is satisfied by a meager set of continuous vector valued observables as the following result shows.
Let $\mathcal M_{inv}(f)$ denote the space of $f$-invariant probabilities.

\begin{mainproposition}\label{p:A}
Let $X$ be a compact Riemannian manifold. There exists a $C^0$-Baire residual subset $\mathfrak{R} \subset \text{Homeo}(X)$
such that the following holds: for every $f\in \mathfrak{R}$ there exists a $C^0$-Baire residual subset 
$ \mathfrak{R} _f\subset C^0(X,\mathbb R^d)$ so that for any $\varphi \in \mathfrak{R} _f$ there exist
$\mu_1,\mu_2\in \mathcal M_{inv}(f)$ such that $\int \varphi \, d\mu_1 \neq \int \varphi \, d\mu_2$.
\end{mainproposition}

As a consequence of the previous results we deduce the following:

\begin{maincorollary}\label{cor:A}
Let $X$ be a compact Riemannian manifold. There exists a $C^0$-Baire residual subset 
$\widehat{ \mathfrak{R}} \subset \text{Homeo}(X) \times C^0(X,\mathbb R^d)$ so that:
\begin{enumerate}
\item 
if $C\subset CR(f)$ is a isolated chain recurrent class,  $(f,\varphi)\in \widehat{ \mathfrak{R}}$ and
$C\cap X_{\varphi,f} \neq \emptyset$ then 
$\htop(f\mid_{C\cap X_{\varphi,f}})=\htop(f\mid_C)$; 
\item if $CR(f)=X$ then $\htop(f\mid_{X_{\varphi,f}})=\htop(f\mid_C)$.
\end{enumerate}
\end{maincorollary}

\subsection{Overview in the proof}

The first ingredient in the proof of Theorems~\ref{thm:C} and \ref{thm:D}
relies on the fact that the dynamics restricted to isolated chain recurrent classes of $C^0$-generic homeomorphisms
satisfies the gluing orbit property.
This will ensure that any connected subset of the pointwise rotation set can be realized by the (pointwise) rotation set
obtained along the orbit of a single point. Such a reconstruction of rotation vectors as the orbit of a single point
is formalized in
Theorems~\ref{thm:A} and \ref{thm:B}. In comparison
with the former, extra difficulties arise from the fact that the dynamics and the observables are not decoupled and the
fact that, in the case of dissipative homeomorphisms,
the chain recurrent classe(s) that concentrate topological pressure
vary as the potential changes.
One could ask whether the Baire generic conclusion of Theorem~\ref{thm:D} could extend to a generic set of points
in the whole chain-recurrent set (or the non-wandering set). For instance, it is easy to construct an Axiom A diffeomorphism
$f$ on $\mathbb S^2$ so that $\Omega(f)=\{p_1\} \cup \Lambda \cup \{p_2\}$, where $p_1$ is a repelling fixed point, $p_2$ is an attracting fixed point and $\Lambda$ is an horseshoe.
Recall that $f$ is an Axiom A diffeomorphism if the set of periodic points is dense in the non-wandering set $\Omega(f)$ and
$\Omega(f)$ is hyperbolic (we refer, for instance, to \cite{Smale} for the construction of such examples). The existence of a filtration for homeomorphisms $C^0$-close to $f$ imply that the Baire generic subset in the statement of Theorem~\ref{thm:D} can only be contained in a neighborhood of the basic piece $\Lambda$ for all $C^0$-close homeomorphisms. Moreover,  the assertion concerning positive entropy seems optimal. Indeed, it may occur
that there exists a unique chain recurrent class of largest positive topological entropy and whose (restricted) rotation set have empty interior or even reduce to a point, in the case of pseudo-rotations. Related constructions include \cite{09ENA,Rees}.

Theorems~\ref{thm:E} relies on the fact that under the specification, or the gluing orbit property,
the space of periodic measures is dense in the space of all invariant measures. Under any of these assumptions,
the generalized rotation set coincides with the rotation set obtained by means of invariant measures, thus it is convex.

Theorems~\ref{thm:A} and \ref{thm:B} provide three distinct measurements of the topological complexity of the set of points with historic behavior. Their proofs use the construction of points with non-convergent Birkhoff averages by exploring the oscillatory
behavior in the Birkhoff averages of points that shadow pieces of orbits that are typical for invariant measures with different
space averages. The existence of such points is granted by the gluing orbit property.

If, on the one hand, the proof of Theorems~\ref{thm:A} and ~\ref{thm:B} are inspired by \cite{Barreira,LW,TMP},
the arguments in the proof of Theorem~\ref{thm:B} is much more challenging and presents novelties on how to
construct a `large amount' of points whose
finite pieces of orbits up to time $n$ have a controlled behavior and that are separated by the dynamics. This is crucial to
estimate topological pressure and metric mean dimension.
While the construction of points with non-convergent behavior can be obtained as a consequence of the gluing orbit property,
it is natural to inquire on the control on the number of such distinct orbits (measured in terms of $(n,\vep)$-separability).
We overcome this issue by selecting of a large amount of orbits that are glued the same (bounded) time. Since this bound depends
on $\vep$, so does the estimates on the number of $(n,\vep)$-separated points with controlled recurrence.
This requires shadowing times to be chosen large in order to compensate the latter. In \cite{DOT} the authors obtain similar flavored results using shadowing. Although both occur properties hold $C^0$-generically there are several examples that satisfy the gluing orbit property and fail to satisfy shadowing, which
justifies our approach.

\section{Preliminaries}\label{sec:prelim}

\subsection{The space of homeomorphisms homotopic to identity \label{shhi}}

Let $X$ be a compact metric space.
Let $\text{Homeo}(X)$ denote the space of homeomorphisms on $X$ endowed with the $C^0$-topology given by the metric
$$d_{C^0}(f, g) = \max\{\sup\{d(f(x), g(x)): x\in X \}, ~\sup\{d(f^{-1}(x), g^{-1}(x)): x\in X \}\}$$ for every  $f, g \in \text{Homeo}(X)$.
Two homeomorphisms $f, g : X \to X$ are \emph{homotopic} if there exists a continuous function
$H : [0,1] \times X \to X$ (homotopy between $f$ and $g$)
such that $H(0,x) = f(x)  \,\text{and}\, H(1,x) = g(x)$ for every $x\in X$.
If $H$ is a homotopy between $f$ and $g$, then it defines a family of continuous functions $H_t : X \to X$ given by $H_t (x) = H(t, x)$.  Two homeomorphisms $f, g : X \to X$ are \emph{isotopic} if there exists a homotopy $H$ between $f$ and $g$ such that for every $t\in [0, 1]$ the map $H_t: X \to X$ is a homeomorphism. It follows  from \cite[Theorem 6.4]{EPS}
 that the previous concepts coincide for homeomorphisms on $\mathbb R^2$. More precisely:

\begin{theorem} \label{epst} If $h$ is a homeomorphism of $\mathbb R^2$ onto itself,
homotopic to the identity then $h$ is isotopic to the identity.
\end{theorem}

Let $\text{Homeo}_{0}^{}(X) \subset \text{Homeo}(X)$ denote the space of homeomorphisms on $X$ homotopic to the identity and let $\text{Homeo}_{0,\lambda}^{}(X)$ be the subspace of $\text{Homeo}_{0}^{}(X)$ formed by the area-preserving homeomorphism ($f$ is area-preserving if $\text{Leb}(f^{-1}(A))=\text{Leb}(A)$ for all $A\subset X$ measurable). In other words, $\text{Homeo}_{0,\lambda}^{}(X):= \text{Homeo}_0^{} (X) \cap \text{Homeo}_\lambda^{} (X)$, where $\text{Homeo}_\lambda^{} (X)$ consisting of area-preserving homeomorphisms. Theorem~\ref{epst} ensures that $\text{Homeo}_{0}^{}(\mathbb T^2) \subset \text{Homeo}_{}^{}(\mathbb T^2)$ is an open set and, consequently, $\text{Homeo}_{0, \lambda}^{}(\mathbb T^2)$ is $C^0$-open in $\text{Homeo}_{\lambda}^{}(\mathbb T^2)$.

\subsection{Rotation sets for homeomorphisms in $\mathbb T^2$}\label{subsec:rotation-sets}

In this subsection we recall briefly some notions and properties of rotation sets (see \cite{MZ,MZ2}
for more details and proofs).
Let $f : X\to X$ be a continuous map and $\varphi : X \to\mathbb R^d~(d\geq 1)$ be a continuous function.
The \emph{rotation set} of $\varphi$,  denoted by $ \rho(\varphi)$, is the set of limits of convergent sequences
$
(\frac{1}{n_{i}}\sum_{i=0}^{n_{i}-1}\varphi (f^{i}(x_i)))_{i=1}^{\infty}
$,
where $\displaystyle  n_{i}\rightarrow \infty$ and  $x_i \in X$.
Given  $x\in X$, let $\mathcal V_\varphi(x)$ denote the accumulation points of the sequence
$
(\frac{1}{n}\sum_{i=0}^{{n}-1}\varphi (f^{i}(x)))_{n\ge 1}
$,
and let
$$
\mathcal V_{\varphi} := \bigcup_{x\,\in\,X}~\mathcal V_{\varphi}(x)
$$
be the \emph{pointwise rotation set} of $\varphi$. In the case that $\mathcal V_{\varphi}(x)=\{v\}$ we say that $v$
is the \emph{rotation vector} of $x$.
Finally, given an $f$-invariant probability measure $\mu$ on $X$ we say that $\int \varphi d\mu$ is the \emph{rotation vector of $\mu$}
and denote it by $\mathcal V_{\varphi}(\mu)$.

\smallskip

In the special case that $X=\mathbb T^2=\mathbb R^2
/ \mathbb Z^2$, $f\in \text{Homeo}_{0}^{}(\mathbb T^2)$, $\pi: \mathbb R^2\to \mathbb T^2$ is the natural projection, $F: \mathbb R^2 \to \mathbb R^2$ a lift for $f$ (ie.  $f\circ \pi = \pi \circ F$), and
the displacement function $\varphi_F: \mathbb R^2 \to \mathbb R^2$ is defined by~$ \varphi_F (\pi(z)) = F(z) - z$, then
$$
\frac{1}{n_{i}}\sum_{i=0}^{n_{i}-1}\varphi_F (f^{i}(\pi(z_i^{})))
	= \frac{1}{n_{i}}\sum_{i=0}^{n_{i}-1} ( F^{i+1}(z_i^{}) - F^{i}(z_i^{}) )
	= \frac{F^{n_i^{}}(z_{i}^{})-z_{i}^{}}{n_{i}^{}}
$$
with $z_i \in \mathbb R^2$ and $n_i\ge 1$.
Using that $\varphi_F$ is constant on $\pi^{-1}(x)$ for every $x\in \mathbb T^2$ it induces a continuous observable
in $\mathbb R^2$, which we still denote by $\varphi_F$ by some abuse of notation.
The \emph{rotation set} of $F$ (denoted by $\rho(F)$) defined in \cite{MZ2} as the limits
of converging sequences
$$
\Big( \frac{F^{n^{}}(z^{})-z^{}}{n^{}} \Big)_{z\in \mathbb R^2, \,n\ge 1}.
$$

Given $x\in \mathbb R^2$, let $\rho(F,x)=\mathcal V_{\varphi_F}(x)$ and $\rho_p(F)=\mathcal V_{\varphi_F}$ denote the \emph{pointwise rotation set} of $F$ along the orbit of $x$ and the \emph{pointwise rotation set}
of $F$ as defined before with respect to the observable $\varphi_F$,
which fit in the previous context.

The rotation set induced by the ergodic probability measures is
$\rho_{erg}(F):=\{\int \varphi_F \, d\mu : \mu \in \mathcal M_{e}(f)\},$
where $\mathcal M_{e}(f)$ denote the set of $f$-invariant and ergodic probability measures
(analogous for $\rho_{inv}(F)$ using the space $\mathcal M_{inv}(f)$ of $f$-invariant probability measures).
We recall that
\begin{equation}\label{eq:rel-rot}
\rho_{erg}(F)\subseteq \rho_{p}^{}(F)\subseteq\rho(F) \subseteq \rho_{inv}(F)
\end{equation}
and that $\rho_{inv}(F)$ is convex.  Moreover, if   $f\in \text{Homeo}_{0}(\mathbb T^2)$ then
$$
\rho(F)=\text{Conv} ~\rho(F) =\text{Conv} (\rho_{p}(F)) = \text{Conv} ~(\rho_{erg}(F))=\rho_{inv}(F)
$$
where $\text{Conv} (K)$ denotes the convex hull of $K$ (see~\cite{MZ}).

\subsection{Shadowing, specification and gluing orbit properties}\label{gop}

The concept of reconstruction of orbits in topological dynamics gained substantial importance for its
wide range of applications in ergodic theory. Among these properties it is worth mentioning the
shadowing, specification and the gluing orbit properties. Throughout this subsection let
$f: X\to X$ be a continuous map on a compact metric space $X$.

First we recall the definition of the shadowing property. Given $\delta  > 0$,
we say that $(x_k)_k$ is a \emph{$\delta$-pseudo-orbit}  for $f$ if $d(f(x_k), x_{k+1}) < \delta$ for every $k\in\mathbb Z$.
If there exists $N>0$ so that $ x_k^{} = x_{k+N}^{}$ for all $k\in \mathbb Z$ we say that $(x_k)_k$ is a \emph{periodic $\delta$-pseudo-orbit}.

\begin{definition}\label{def:shadow}
We say that $f$ satisfies the (periodic) \emph{shadowing property} if for any $\vep > 0$ there exists $\delta>0$ such that
for any (periodic) $\delta$-pseudo-orbit $(x_k)_k$ there exists $y\in X$ satisfying $d(f^{k}(y), x_{k})< \vep $ for all $k\in \mathbb Z$.
\end{definition}

Pseudo-orbits are also a fundamental tool to decompose the ambient space according to classes.
Given a homeomorphism $f\in \text{Homeo}(X)$,
we say that $x\sim y$ if for any $\de>0$ there exists a $\delta$-pseudo-orbit 
$(x_k)_{1\le k \le n}$ so that $x_1=x$ and $x_2=y$. A point is called \emph{chain recurrent} if $x\sim x$, and we denote
by $CR(f)$ the chain recurrent set. Notice that $\sim$ is an equivalence relation.  
A \emph{chain recurrent class} $C\subset CR(f)$ is a maximal subset so that $x\sim y$ for every $x,y\in C$. 
It is known that the non-wandering set $\Omega(f)$
and the chain recurrent set $CR(f)$ of a $C^0$-generic homeomorphism $f$ coincide (cf. \cite[Theorem~1]{PPSS}). 
Finally,  we say that a chain recurrent class $C\subset CR(f)$ is \emph{isolated} if $\dist_H(C, CR(f)\setminus C)>0$,
where $\dist_H(\cdot, \cdot)$ denotes the Hausdorff distance between sets.

\medskip
The specification property, introduced by Bowen \cite{Bo71},
roughly means that an  arbitrary number of pieces of orbits can be ``glued together" to obtain a real orbit that shadows the previous ones with a prefixed number of iterates in between.  Moreover, it configures itself as an indicator of chaotic behavior (e.g. it implies the dynamics to have positive topological entropy).

\begin{definition}\label{def:spec}
We say that $f$ satisfies the \emph{specification property }  if
for any $\vep > 0$ there exists an integer $m = m(\vep) \geq 1$ so that for any points $x_1^{}, x_2^{}, \dots , x_k^{} \in X$ and for any positive integers $n_1^{}, \dots , n_k^{}$ and $0 \leq p_1^{}, \dots , p_{k-1}^{}$ with $p_i \geq m(\vep)$ there exists a point $y \in X$ such that $\displaystyle d(f^j (y), f^j (x_1^{}))\leq \vep$ for every $0 \leq j \leq n_1^{}$ and
$$
d(f^{j+n_1^{}+p_1^{}+ ~\dots~ +n_{i-1}^{}+p_{i-1}^{}} (y), f^j (x_i^{}))\leq \vep
$$
\noindent for every $2 \leq i \leq k$ and $0 \leq j \leq n_i^{}$.
\end{definition}

Finally, the gluing orbit property, introduced in \cite{BV},
bridges between completely non-hyperbolic dynamics (equicontinuous and minimal dynamics \cite{TBTV,Sun}) and uniformly hyperbolic dynamics (see e.g. \cite{BV}). Both of these properties imply on a rich structure on the dynamics and the space of invariant measures (see e.g. \cite{DKS,TBTV}).

\begin{definition}\label{def:gluing}
We say that $f$ satisfies the \emph{gluing orbit property} if for any $\vep > 0$ there exists an integer $m = m(\vep) \geq 1$ so that for any points $x_1^{}, x_2^{}, \dots , x_k^{} \in X$ and any positive integers $n_1^{}, \dots , n_k^{}$ there are $0 \leq p_1^{}, \dots , p_{k-1}^{} \leq m(\vep)$ and a 
point $y \in X$ so that $\displaystyle d(f^j (y), f^j (x_1^{}))\leq \vep$ for every $0 \leq j \leq n_1^{}$ and
$$
d(f^{j+n_1^{}+p_1^{}+ \dots +n_{i-1}^{}+p_{i-1}^{}} (y), f^j (x_i^{}))\leq \vep
$$
for every $2 \leq i \leq k$ and $0 \leq j \leq n_i^{}$.
If, in addition, $y\in X$ can be chosen periodic with period $\sum_{i=1}^k (n_i+p_i)$
for some $0 \leq p_{k}^{} \leq m(\vep)$ then we say that $f$ satisfies the \emph{periodic gluing orbit property}.
\end{definition}

It is not hard to check that irrational rotations satisfy the gluing orbit property \cite{TBTV}, but fail to satisfy the shadowing
or specification properties. Partially hyperbolic examples exhibiting the same kind of behavior have been constructed in \cite{BTV2}.

\begin{remark}\label{def:gluing-local}
It is clear that the specification property implies the gluing orbit property, which implies transitivity. It will be useful to consider the (periodic) gluing orbit property on compact invariant subsets $\Gamma$, in which case we demand only Definition~\ref{def:gluing} to hold for every small $\vep$ but we require the shadowing point $z$ to belong to $\Gamma$.
\end{remark}

\subsection{Pressure, entropy and mean dimensions}\label{sec:entropy-mmd}

In this subsection we recall two important measurements of topological complexity, namely the concepts of topological entropy and metric mean dimension, and introduce a relative notion of the later. Our interest in the second notion is that, while a dense set of homeomorphisms on a compact Riemannian manifold have positive and finite topological entropy (by denseness of $C^1$-diffeomorphisms) it is known that typical homeomorphisms may have infinite topological entropy. In opposition, metric mean dimension is always bounded by the dimension of the compact manifold and can be seen as a smoothened measurement of topological complexity
as we now detail.

\subsubsection*{Topological pressure}
Let $(X,d)$ be a compact metric space and $\psi \in C^0(X,\mathbb R)$. Given $\vep > 0$~and $n \in \mathbb N$, we say that $E\subset X$ is $(n, \vep)$-separated if for every $x \neq y\in E$ it holds
that $d_n (x, y)> \vep$, where $d_n (x, y) = \max\{d(f^j (x), f^j (y)); j = 0, \dots, n - 1\}$ is the Bowen's distance. The sets $B_{n}(x, \vep) =\{ y\in X: d_n^{}(x, y) < \vep\}$ are called Bowen dynamic balls. 
The \emph{topological pressure} of $f$ with respect to $\psi$ is defined by
\begin{eqnarray*}
\displaystyle P_{top}(f,\psi) = \lim_{\vep \to 0} \limsup_{n \to \infty} \frac{1}{n} \log \sup_{E} \sum_{x\in E} e^{S_n\psi(x)},
\end{eqnarray*}
where $S_n\psi(x)=\sum_{j=0}^{n-1} \psi(f^j(x))$ and the supremum is taken over every $(n,\vep)$-separated sets $E$ contained in $X$. In the case that $\psi\equiv 0$, if $s(n,\vep)$ denotes the maximal cardinality of a $(n, \vep)$-separated subset of $X$, then the \emph{topological entropy} is defined by
\begin{eqnarray*}
\displaystyle h_{top}(f) = \lim_{\vep \to 0} \limsup_{n \to \infty} \frac{1}{n} \log s(n,\vep).
\end{eqnarray*}
The previous notion does not depend on the metric $d$ and is a topological invariant. Moreover, by the classical variational principle for the pressure, it holds that
$
P_{top}(f, \psi) = \sup\{h_\mu (f) + \int \psi d\mu : \mu\in \mathcal M (f)\}.
$
However, the topological entropy of $C^0$-generic homeomorphisms on a closed manifold of dimension at least two
is infinite \cite{Yano} (the same holds for the topological pressure as a consequence of the variational principle),
in which case neither the topological entropy nor topological pressure  can distinguish such dynamics.

\subsubsection*{Topological and metric mean dimension}
Gromov \cite{Gro} proposed an invariant for dynamical systems called \emph{mean dimension}, that was further studied by Lindenstrauss and Weiss \cite{LindW}. The upper and lower \emph{metric mean dimension}, which may depend on the metric, are defined in \cite{LindT,LindW} by
\begin{eqnarray*}
\displaystyle \overline{\text{mdim}}(f)~ = \lim_{\vep \to 0} \frac{\overline{\lim}_{\substack{n \to \infty}} \frac{1}{n} \log s(n,\vep)}{-\log \vep}
\end{eqnarray*}
and
\begin{eqnarray*}
\displaystyle \underline{\text{mdim}}(f)~ = \lim_{\vep \to 0} \frac{\underline{\lim}_{\substack{n \to \infty}} \frac{1}{n} \log s(n,\vep)}{-\log \vep},
\end{eqnarray*}
respectively. Observe that the latter quantitities are only meaningful whenever $f$ has infinite topological entropy.
In the case that the metric space satisfies a tame growth of covering numbers, the the metric mean dimension satisfies a variational principle involving a concept of measure theoretical mean dimension (cf. \cite{LindT}).

\subsubsection*{Relative metric mean dimension}

Since we aim to describe the topological complexity of (not necessarily compact) $f$-invariant subsets we now introduce a concept of relative metric mean dimension using a Carath\'eodory structure. Let $Z \subset X$ be an $f$-invariant Borel set. Given $s \in \mathbb R$ and $\psi\in C^0(X,\mathbb R)$ define
$$
\displaystyle Q(Z, \psi, s, \Gamma) = \!\!\!\sum_{B_{n_i}^{} (x_i^{}, \vep)\in \Gamma} \!\! e^{-s\, n_i^{}\, + S_{n_i}\psi(B_{n_i}^{} (x_i^{}, \vep))}
\;
\text{and}
\;
\displaystyle M(Z, \psi, s, \vep, N) \!= \! \inf_{\Gamma}\left\{ Q(Z, \psi, s, \Gamma)\right\},
$$
where $S_{n_i}\psi(B_{n_i}^{} (x_i^{}, \vep)):=\sup_{x\in B_{n_i}^{} (x_i^{}, \vep)}\; \sum_{k=0}^{n_i^{}-1}\psi(f^k(x)))$
and where the infimum is taken over all countable collections $\Gamma = \left\{B_{n_i} (x_i, \vep)\right\}_i$ that cover $Z$ and so that $n_i \geq N$. Since the function $M(Z, \psi, s, \vep, N)$ is non-decreasing in $N$ the limit $\displaystyle m(Z, \psi,  s, \vep) = \lim_{N\to \infty} M(Z, \psi, s, \vep, N)$ does exist.
Then let
\begin{eqnarray*}
P_{Z}(f, \psi, \vep) = \inf\{s \in \mathbb R \colon   m(Z, \psi,  s, \vep)  = 0\}
	= \sup\{ s \in \mathbb R \colon   m(Z, \psi,  s, \vep)  = \infty\}.
\end{eqnarray*}
The existence of $\displaystyle P_{Z}(f, \psi, \vep)$ follows by the Carath\'eodory structure
\cite{PesinB}.
The (relative) \emph{topological pressure } of $f$ on $Z$ with respect to $\psi$ is defined by
$$
\displaystyle P_{Z}(f, \psi) = \lim_{\vep \to 0} P_{Z}(f, \psi, \vep).
$$
We set $h_{Z}(f,\vep)=\displaystyle P_{Z}(f, 0,\vep)$ for every $\vep>0$ and define the \emph{relative entropy} of $f$ on $Z$ by $h_{Z}(f)=\displaystyle P_{Z}(f, 0)$ (which corresponds to the potential $\psi\equiv 0$).

The upper and lower \emph{relative metric mean dimension} of $Z$ are
\begin{eqnarray*}
\displaystyle \overline{\text{mdim}}_{Z}^{}(f) = \overline{\lim}_{\vep \to 0} \frac{h_{Z}(f, \vep)}{-\log \vep}
	\quad \text{and} \quad
\displaystyle \underline{\text{mdim}}_{Z}^{}(f) = \underline{\lim}_{\vep \to 0} \frac{h_{Z}(f, \vep)}{-\log \vep}
\end{eqnarray*}
respectively. If the previous limits do exist we represent simply by $\text{mdim}_{Z}^{}(f)$ and refer to this
as the relative metric mean dimension of $Z$.

\begin{definition}
We say that the $f$-invariant subset $Z\subset X$ has full topological entropy if $h_Z(f) =\htop(f)$. We say that the $f$-invariant subset $Z\subset X$ has full metric mean dimension if $\underline{mdim}_Z(f)=\underline{mdim} (f)$ and $\overline{mdim}_Z(f)=\overline{mdim} (f)$.
\end{definition}

\begin{remark} \label{remhtop=hX}
If $f : X \to X$ is a continuous map on a compact metric space and $\psi\in C^0(X,\mathbb R)$ then $P_{top}^{}(f,\psi) = P_{X}^{}(f,\psi)$. Moreover, if the limits exist and coincide
then $\text{mdim}_{X}^{}(f) = \text{mdim}~(f).$ This follows from the fact that $h_{X}^{}(f, \vep) =  h_{top}^{}(f, \vep)$
for any $\vep>0$, which can be read from the proof of \cite[Proposition 4 ]{PP} (actually in \cite{PP} the authors use the definition of entropy using coverings and prove that $h_{X}^{}(f, \cU) =  h_{top}^{}(f, \cU)$ for every open cover $\cU$).
\end{remark}

\begin{remark}
The notion of Hausdorff dimension also involves a Carath\'eodory structure, associated to the function
$
Q(Z, s, \Gamma) = \sum_{B_{n_i} (x_i, \vep)\in \Gamma} \text{diam}(B_{n_i} (x_i, \vep))^s
$
(see \cite[Section~6]{PesinB}). Inspired by \cite{Barreira} we expect that for continuous and transitive maps on the interval (these satisfy the gluing orbit property) the  set of points with historic behavior is either empty or to have Hausdorff dimension equal to one. We do not claim or prove this fact here.
\end{remark}

We use the following generalization of Katok's formula for pressure:

\begin{proposition} \cite[Proposition 2.5]{TMP} \label{TMP} Let $(X, d)$ be a compact metric space, $f$ be a continuous map on $X$
and $\mu$ be an $f$-invariant, ergodic  probability. Given $\vep > 0$, $\gamma\in (0, 1)$ and  $\psi\in C^0(X,\mathbb R)$ set
$
N^{\mu} (\psi, \gamma, \vep, n) = \inf_E\sum_{x\in E}\exp\big\{\sum_{i = 0}^{n-1} \psi(f^{i} (x))\big\},
$
where the infimum is taken over all sets $E$  that
$(n, \vep)$-span a set $Z$ with $\mu(Z) \geq 1-\gamma$. Then
\begin{eqnarray*}
h_\mu^{} (f) +\int \psi d\mu = \lim_{\vep \to 0}\liminf_{n\to \infty} \frac{1}{n} \log N^\mu (\psi, \gamma, \vep, n).
\end{eqnarray*}
\end{proposition}

\begin{remark}\label{defvar}
Given $\vep>0$ and $\varphi\in C^0(X,\mathbb R^d)$, the variation in balls of radius $\vep$ is
$$
\displaystyle \var(\varphi, \vep) = \sup \{\mid \varphi(x) - \varphi(y)\mid : d(x, y)<\vep\}.
$$
Since $X$ is compact then  $\var(\varphi, \vep)\to 0$ as $\vep \to 0$. As $\varphi : X \rightarrow \mathbb R^d$ is continuous (hence uniformly continuous) and $f$ is continuous then for every $\vep >0$ there exists $\delta>0$ such that $\| \frac{1}{n}\sum_{i=0}^{n - 1}\varphi(f^i (x)) - \frac{1}{n}\sum_{i=0}^{n - 1}\varphi(f^i (y))\| < \vep$ whenever $d_n (x, y)<\delta$.
\end{remark}

\section{The set of points with non-trivial pointwise rotation set}\label{sec:homotopic1}

The main goal of this section is to prove Theorems~\ref{thm:C} and \ref{thm:D}, concerning on the set of points in $\mathbb T^2$
with non-trivial pointwise rotation set for typical homeomorphisms.

\subsection{Continuous maps with the gluing orbit property}\label{subsec:gop}

Here we prove the genericity of the gluing orbit property on chain recurrent classes 
with a dense set of periodic orbits, 
a result of independent interest inspired by \cite{BTV}. 

\begin{proposition}\label{lem0}
Let $X$ be a compact Riemannian manifold of dimension at least $2$. Assume that $ f \in \text{Homeo}^{}(X)$ 
has the periodic shadowing property.
If $\Gamma \subset CR(f)$ is a isolated chain recurrent class 
then $f\mid_\Gamma$ satisfies the periodic gluing orbit property.
\end{proposition}

\begin{proof}
By the periodic shadowing property, periodic points are dense in isolated chain recurrent classes. 
Thus $\Gamma$ is a compact set with a dense set of periodic points. Given $\de>0$, let 
$\mathfrak L = \{\theta_1, \theta_2, \dots , \theta_m\} \subset \Gamma \cap Per (f)$ 
be a maximal $\de$-separated subset of $\Gamma$.
We know that $m=m(\de)=\#\mathfrak L <\infty$ by the compactness of $\Gamma$. Moreover, if $\pi(\theta_i)\ge 1$ denotes the
prime period of the periodic point $\theta_i$ then it is not hard to check that for any points 
$x, y \in \Gamma$ there exists a $\delta$-pseudo orbit $(x_i^{})_{i=1, \dots , n}$ 
so that $x_0^{} = x, x_n^{} = y$ and  $n \le m \cdot \max_{1\le i\le m} \pi(\theta_i)$.
Indeed, choose $\{\theta_{i_1}, \theta_{i_2}, \dots, \theta_{i_s}\} \subset \mathfrak L$
with $s\le m$ so that $d(x, \theta_{i_1})<\de$, $d(\theta_{i_s},y)<\de$ and $d(\theta_{i_j}, \theta_{i_{j+1}})<\de$
for every $1\le j\le s-1$ and take the $\de$-pseudo orbit 
$$
\big\{x^{}, \theta_{i_1}, f(\theta_{i_1}), \dots , f^{\pi(\theta_{i_1})-1}(\theta_{i_1}), 
	\theta_{i_2}, f(\theta_{i_2}), \dots , f^{\pi(\theta_{i_2})-1}(\theta_{i_2}), \dots
$$
$$
\dots, \theta_{i_s}, f(\theta_{i_s}), \dots , f^{\pi(\theta_{i_s})-1}(\theta_{i_s}), y \big\}
$$
connecting $x$ to $y$.

We claim that $f\mid_\Gamma$ satisfies the periodic gluing orbit property.
Take an arbitrary $0<\vep<\frac12 \dist_H(\Ga, CR(f)\setminus \Ga)$ and 
let $\delta=\delta(\vep) > 0$ be given by the periodic shadowing property.

Consider arbitrary points $x_1, x_2, \dots ,x_k \in \Gamma$ and integers $n_1^{}, \dots , n_k^{} \ge  0$. 
The previous argument ensures that, for every $1\le s \le k$ there exists a $\delta$-pseudo orbit $(y_i^{s})_{i=0, \dots , l_{s}}$ connecting the point
$f^{n_{s-1}}(x_{s-1})$ and $x_{s}$ and  a $\delta$-pseudo orbit $(y_i^{k+1})_{i=0, \dots , l_{k+1}}$ connecting the point
$f^{n_{k}}(x_{k})$ and $x_{1}$, all formed by at most $K:=m \cdot \max_{1\le i\le m} \pi(\theta_i)$ points.
Notice that $K$ depends only on $\mathfrak L$ and $\delta$. 
Hence we may consider the $\delta$-pseudo-orbit  $(x_i^{})_i^{}$ connecting $x_1^{}$ to itself 
defined by
$$
\big\{x_1^{}, f(x_1^{}), \dots , f^{n_1^{}-1}(x_1^{}), y_0^{1},  y_1^{1}, \dots , y_{l_{1}^{}-1}^{1}, x_2, f(x_2^{}), \dots , f^{n_2^{}-1}(x_2^{}),
$$
$$
y_0^{2},  y_1^{2}, \dots , y_{l_{2}-1}^{2},   \dots , x_{k}^{}, f(x_{k}^{}), \dots , f^{n_{k}^{}}(x_{k}^{}), y_0^{k+1},  y_1^{k+1}, \dots  y_{l_{k+1}}^{k+1}, x_1\big\}.
$$
Using the periodic shadowing property for $f$ there exists a periodic point $z\in X$ so that $d(f^{j}(z), f^{j}(x_1^{})) < \vep$  for every $0\le  j\le n_1^{}$ and
$$
d(f^{j + p_{1}^{} + n_1^{} +  \dots + p_{i-1}^{} + n_k^{}}(z), f^{j}(x_i^{})) < \vep , ~~\forall ~~i\in \{2, \dots , k\},~~\forall~~j\in \{0, 1, \dots , n_i^{}\}
$$
where each $p_{s}$ is bounded above by $K$. The choice of $\vep$ ensures that $z\in \Gamma$, hence  $f\mid_\Gamma$ satisfies the periodic gluing orbit property.
\end{proof}

\begin{corollary}\label{lem1}
Let $X$ be a compact Riemannian manifold of dimension at least $2$.
There exists a Baire residual subset $\mathcal R_0\subset \text{Homeo}_{0}^{}(X)$  so that if $f\in \mathcal R_0$ and $\Gamma \subset \Omega(f)$ is a isolated chain recurrent class then the restriction $f\mid_\Gamma$ satisfies the periodic gluing orbit property.
\end{corollary}

\begin{proof}
It follows from \cite{CMN,pipa} that there exists a residual subset  $\widetilde{\mathcal R_0}\subset\text{Homeo}(X)$ such that every $f\in \widetilde{\mathcal R_0}$ has the periodic shadowing property, 
and $\overline{\text{Per}(f)} = \Omega (f)=CR(f)$.
The result is now a direct consequence of Proposition~\ref{lem0}.
\end{proof}

\subsection{Volume preserving homeomorphisms}\label{subsec:conservative}

Our starting point for the proof of Theorem~\ref{thm:C} is that
specification is generic among volume preserving homeomorphisms.
More precisely, for any compact Riemannian manifold $M$ of dimension at least $2$, there exists a residual subset $\mathcal R_2 \subset \text{Homeo}_\lambda(M)$ such that every homeomorphism in $\mathcal R_2$ satisfies the specification property \cite{GL}. Together with the fact that  $\text{Homeo}_{0, \lambda}^{}(\mathbb T^d)$ is open in $\text{Homeo}_{\lambda}^{}(\mathbb T^d)$ this ensures:

\begin{corollary} \label{cor:Hgluing} There is a residual $\mathcal R_3 \subset \text{Homeo}_{0,\lambda}^{}(\mathbb T^d)$ such that every $f\in \mathcal R_3$ satisfies the specification property (hence the gluing orbit property).
\end{corollary}

Given $f\in \text{Homeo}_{0}^{}(\mathbb T^2)$ recall that $\rho(f)$ is called \emph{stable} if there exists $\delta>0$ so that $\rho(g) = \rho(f)$ for every $g\in \text{Homeo}_{0}^{}(\mathbb T^2)$ so that $d_{C^0}^{} (f, g) < \delta$.

\begin{theorem}\cite[Theorem 1]{GK} \label{lem2}  The set of all homeomorphisms with a stable rotation set  is open and dense
set $\mathcal O \subset \text{Homeo}_{0}^{}(\mathbb T^2)$.
Moreover, the rotation set of every such homeomorphism is a convex polygon with rational vertices, and in the area-preserving setting this polygon has nonempty interior.
\end{theorem}

We are now in a position to prove Theorem \ref{thm:C}.

\begin{proof}[Proof of Theorem \ref{thm:C}]
Let $F: \mathbb R^2\to\mathbb R^2$ be a lift of $f$ and consider the observable (displacement function) $\varphi_F : \mathbb T^2\to\mathbb R^2$ given by   $\varphi_F (x) = F(\widetilde{x}) - \widetilde{x}$, where $\pi(\widetilde{x})=x$. Since
$$
\frac1n\sum_{j=0}^{n-1} \varphi_F( f^j(x)) = \frac{F^n (\widetilde{x}) - \widetilde{x}}{n}
$$
then $\rho(F,x)$ coincides with the accumulation points of $(\frac1n\sum_{j=0}^{n-1} \varphi_F ( f^j(x)))_{n\ge 1}$.
Hence, the set of points with non-trivial pointwise rotation set of $x$ can be defined by
$
\mathbb F_{f}^{} := \{x \in \mathbb T^2 :  \rho(F, \widetilde{x}) ~\text{is~not~trivial, where }  \pi(\widetilde{x})=x\}.
$
Take the residual subset
$
\mathfrak{R}_1 := \mathcal{R}_3\cap \mathcal{O} \subset \text{Homeo}_{0,\lambda}(\mathbb T^2).
$
We claim that $\mathbb F_{f}$ is residual in $\mathbb T^2$ for every $f\in \mathfrak R_1$. Indeed, any $f\in \mathfrak R_1$ satisfies the gluing orbit property and, if $F$ is  a lift of $f$,
$\text{int} \;\rho(F)\neq\emptyset$. The latter ensures that $\varphi_{F}^{} \not\in \overline{Cob}$ (recall Lemma \ref{obsat}) and $X_{\varphi_{F}^{}, f}\neq\emptyset$. Theorem \ref{thm:C} is now a consequence of Theorems \ref{thm:A} and \ref{thm:B}.
\end{proof}

\subsection{Dissipative homeomorphisms}\label{subsec:dissipative}

In order to prove Theorem~\ref{thm:D} consider the set
$\mathcal A := \{f \in \text{Homeo}_{0}^{}(\mathbb T^2) : \text{int}\; \rho(F)\neq\emptyset\}$, which does not depend on the lift $F$. Misiurewicz and Ziemian proved that $ \mathcal A$ is open in $\text{Homeo}_{0}^{}(\mathbb T^2)$ \cite[Theorem~B]{MZ}. The following useful results are due the Libre and Mackay \cite{LM}.

\begin{theorem} \cite[Theorem~1]{LM} \label{proprs4}
If $f\in \text{Homeo}_{0}(\mathbb T^2)$ and $F$ is a lift of $f$ then the following hold:
(i) if $\rho(F)$ has nonempty interior then $f$ has positive topological entropy; and ~(ii)  if $\Delta\subset \rho(F)$ is a polygon whose vertex are given by the rotation vectors of
(finitely many) periodic points of $f$ then for any compact connected $D\subset \Delta$ there exists point $x\in \mathbb T^2$ and $\tilde{x}\in \pi^{-1}(x)$ so that $\rho(F, \tilde{x})=D$.
\end{theorem}

This result is enough to prove the following:

\begin{lemma}\label{lem6565}
Take $f \in \cA$ and let $F$ be a lift of $f$. There exists a chain recurrent class  $\Gamma \subset \Omega (f)$ such that 
$\rho(F\mid_{\pi^{-1}(\Gamma)})$ has non-empty interior. In particular $h_{top}(f\mid_\Gamma)>0$.
\end{lemma}

\begin{proof}
By Franks \cite{1989F} all rational points in the interior of $\rho(F)$ are realizable by periodic point of $f$.
In particular, given a small disk $ D\subset \text{int} \;\rho(F)\neq\emptyset$, there is $x\in \mathbb T^2$ and $\tilde x\in \pi^{-1}(x)$ such that $\rho(F, \tilde{x})=D$ (by item (ii) in Theorem \ref{proprs4}). In consequence $D \subset \rho(F\mid_\Gamma)$ where $\Gamma$ denotes the  chain recurrent class of $f$ containing the point $x$. 
The previous argument shows that there is a chain recurrent class  $\Gamma \subset \Omega (f)$ such that $\rho(F\mid_\Gamma)$ has non-empty interior.  Now Theorem \ref{proprs4} item (i) implies the conclusion of the lemma.
\end{proof}

\begin{proof} [Proof of Theorem~\ref{thm:D}]
Let $\mathcal R_0$ be given by Corollary~\ref{lem1} and take the residual subset $\mathfrak{R}_2 = \mathcal R_0 \cap \mathcal A$.
The first statement in the theorem corresponds to Lemma~\ref{lem6565}.

Now, assume that $ \Gamma$ is a isolated chain recurrent class such that $\rho(F\mid_{\pi^{-1}(\Gamma)})$ has non-empty interior.
Corollary~\ref{lem1} ensures that $f\mid_\Gamma$ satisfies the periodic gluing orbit property.
Then, since $\text{int}\, \rho(F\mid_{ \Gamma}) \neq \emptyset$, the 
displacement function $\varphi_{F}=F-Id$ is not accumulated by functions cohomologous to a vector on $\pi^{-1}(\Gamma)$.
Theorems  \ref{thm:A} and ~\ref{thm:B} imply that $X_{\varphi_{F}^{}, f} \cap { \Gamma}$ is Baire residual and
has full topological entropy and full metric mean dimension in the chain recurrence class $ \Gamma$, proving the theorem.
\end{proof}

\section{Rotation sets on $\mathbb T^d$ are generically convex}
\label{sec:homotopic2}

The main goal of this section is to prove Theorem~\ref{thm:E}.
If $p\in \mathbb T^d$ is a periodic point of prime period $k\ge 1$ (with respect to $f$) we denote by $\mu_p:=\frac1k \sum_{j=0}^{k-1} \delta_{f^j(p)} $ the \emph{periodic measure} associated to $p$.
We will use the following:

\begin{lemma}\label{lemma:per-measures}
Assume that $\Lambda\subset \mathbb T^d$ is a compact $f$-invariant set. If $f\mid_\Lambda$ satisfies the periodic gluing orbit property then periodic measures are dense in $\mathcal M_1(f\mid_\Lambda)$ (in the weak$^*$ topology).
\end{lemma}

\begin{proof}
The proof is a simple modification of the arguments in \cite{Sig1} (where it is considered the case where $f$ satisfies the specification property). We will include a brief sketch for completeness.

Let $(\psi_n)_{n\ge1}$ be countable and dense in $C^0(\Lambda,\mathbb R)$ and consider the metric $d_*$ on $\mathcal M(\Lambda)$ given by $d_*(\nu,\mu)= \sum_{n\ge 1} \frac1{2^n} \big| \int \psi_n d\nu - \int \psi_n d\mu \big|$. This metric is compatible with the weak$^*$ topology in $\mathcal M(\Lambda)$.
The compactness of $\mathcal M_1(f\mid_\Lambda)$ and the ergodic decomposition theorem, ensures that for any $\eta \in \mathcal M_1(f\mid_\Lambda)$ and $\zeta>0$ there exists a probability vector $(\alpha_i)_{1 \leq i \leq k}$ and ergodic measures $(\eta_i)_{1 \leq i \leq k}$ so that $d_*(\eta, \hat \eta)<\zeta/2$, where
$\hat \eta= \sum_{i=1}^k \alpha_i \eta_i.$ It is enough to construct a periodic point $p \in \Lambda$ such that $d_*(\mu_p,\hat \eta) <\zeta/2$. By definition of weak$^*$ topology, one can choose $\vep>0$ so that if $d_n(x,y)<\vep$ then $d_*(\frac1n \sum_{j=0}^{n-1} \delta_{f^n(x)},\frac1n \sum_{j=0}^{n-1} \delta_{f^n(y)}) <\zeta/10$. Let $m(\vep)>0$ be given by the gluing orbit property. Choose $N\ge 1$ large and for any $1\le i \le k$:
\begin{itemize}
\item pick $x_i \in \Lambda$ so that
$d_*(\frac1n \sum_{j=0}^{n-1} \delta_{f^n(x_i)}, \eta_i) <\zeta/10$
for every $n\ge N$,
\item let $n_i \ge N$ be so that
$
\big|\frac{n_i}{\sum_{j=1}^k  n_j} -\alpha_i \big| < \frac\zeta{10 k}
$
and
$
\frac{(k+1) m(\vep)}{\sum_{j=1}^k  n_j}
	\le   \frac{\zeta}{10}.
$
\end{itemize}
By the periodic gluing orbit property there are  positive integers $0\le m_i\le m(\vep)$ and a periodic point $p\in \Lambda$ of
period $\kappa(p)= \sum_{i=1}^k (n_i+m_i)$ satisfying
$$
d_{n_i}(f^{\sum_{j<i} (n_j+m_j)}(p), x_i) <\vep
\quad\text{for every $1\le i \le k$.}
$$
Then,  by triangular inequality, it is not hard to check that
\begin{align*}
 d_*(\mu_p,\hat \eta)
	& \le
	d_*\Big( \frac1{\kappa(p)} \sum_{i=1}^{\kappa(p)} \delta_{f^i(p)},
	\frac1{\kappa(p)}   \sum_{i=1}^k  \sum_{j=0}^{n_i-1} \delta_{f^j(x_i)}
	\Big) \\
	& +
	d_*\Big(\frac1{\kappa(p)}   \sum_{i=1}^k  \sum_{j=0}^{n_i-1} \delta_{f^j(x_i)},
	\frac{1}{\sum_{j=1}^k  n_j}  \sum_{i=1}^k  \sum_{j=0}^{n_i-1} \delta_{f^j(x_i)}
	\Big) \\
	& +
	d_*\Big(   \frac{1}{\sum_{j=1}^k  n_j}  \sum_{i=1}^k  \sum_{j=0}^{n_i-1} \delta_{f^j(x_i)}  , \sum_{i=1}^k  \frac{\alpha_i}{n_i} \sum_{j=0}^{n_i-1} \delta_{f^j(x_i)}
	\Big)\\
	& +
	d_*\Big( \sum_{i=1}^k  \frac{\alpha_i}{n_i} \sum_{j=0}^{n_i-1} \delta_{f^j(x_i)},  \sum_{i=1}^k  {\alpha_i} \eta_i \Big)
	\\
	& \le  \frac{2\,k\,m(\vep)}{\kappa(p)}
	+\frac{3\zeta}{10}
	\le \frac\zeta2,
\end{align*}
which proves the lemma.
\end{proof}

\medskip

\begin{proof}[Proof of Theorem~\ref{thm:E}]

Let $d\ge 2$ be an integer and let $\mathfrak R_3:=\mathcal R_3$  be the $C^0$-residual subset in $\text{Homeo}_{0,\lambda}^{}(\mathbb T^d)$ formed by homeomorphisms with the specification property (cf. Corollary~\ref{cor:Hgluing}).
Given $f\in \mathcal R_3$ and a lift $F$ recall that
$$
\rho_{erg}(F)\subseteq \rho_{p}^{}(F)\subseteq\rho(F) \subseteq \rho_{inv}(F),
$$
and that $\rho_{inv}(F)$ is convex.

We claim that $\rho(F)\supseteq \rho_{inv}(F)$ for every lift $F$ of a homeomorphism $f\in \mathcal R_3$.
Take an arbitrary $v\in \rho_{inv}(F)$ and $\eta\in \mathcal M_{inv}(f)$ so that $v=\int \varphi_F \, d\eta$. 
By specification, there exists a sequence $(p_n)_n$ of periodic points so that $\mu_{p_n} \to \eta$
as $n\to\infty$ (cf. \cite{DKS}). In particular, since $\varphi_F$ is continuous, if $\kappa_n \ge 1$ denotes the prime period of $p_n$
and $\tilde p_n\in \pi^{-1}(p_n)$ then
$$
\frac{F^{\kappa_n}(\tilde p_n )- \tilde p_n}{\kappa_n}
	= \frac1{\kappa_n} \sum_{j=0}^{\kappa_n-1} \varphi_F (f^j(p_n)) 	 \to \int \varphi_F \, d\eta = v
	\quad\text{as} \; n\to\infty.
$$
This ensures that $v\in \rho(F)$. Therefore  $\rho(F) = \rho_{inv}(F)$ is convex, which proves item (1) in the theorem.

The proof of item (2) is completely analogous, using the restriction of the rotation set to each isolated chain recurrent class 
instead of the generalized rotation set, Corollary~\ref{lem1} instead of Corollary~\ref{cor:Hgluing} and taking $\mathfrak R_4:=\mathcal R_0$.
\end{proof}

\color{black}

\section{The set of points with historic behavior}\label{sec:irregular}

The main goal of this section is to prove Theorems~ \ref{thm:A} and \ref{thm:B}, which claim that the set of points with historic behavior
for continuous maps with the gluing orbit property is topologically large. Actually, this is established by means of three different
measurements of topological complexity:  Baire genericity, full topological entropy and full metric mean dimension.
The arguments involved in the proofs of Theorems~\ref{thm:A} and \ref{thm:B} are substantially different and their proofs occupy Subsections~\ref{subsec:Baire} and \ref{subsec:htopmmd}, respectively.

\subsection{Baire genericity of historic behavior}\label{subsec:Baire}

This subsection is devoted to the proof of Theorem~\ref{thm:A}, whose strategy is strongly inspired by \cite{BLV,LW}. The differences lie on the fact that, due to the higher dimensional features of observables, we need to restrict to connected subsets in the set of all accumulation vectors, and that we have transition time functions instead of a determined time to shadow pieces of orbits
(see Remark~\ref{rmk:differencesG} below).

Let $f : X \rightarrow X$ be a continuous map with the gluing orbit property on a compact metric space $X$ and let
$\varphi : X \rightarrow \mathbb R^d$ be a continuous function so that $\mathcal V_\varphi$ is non-trivial.
Let $\Delta \subset \mathcal V_\varphi$ a non-trivial connected set, define
\begin{equation}\label{eqxD}
X_\Delta
	:=\Big\{x\in X \colon  \Delta\subseteq \mathcal V_\varphi(x)  \Big\}.
\end{equation}
Theorem~\ref{thm:A} will be a consequence of the following characterization of the irregular set for vector valued observables 
and maps with gluing orbit property, thus extending previous similar results for real valued observables and maps with
specification.

\begin{proposition}\label{prop:irreg}
Let $X$ be a compact metric space, $f \in \text{Homeo}(X)$ satisfy the gluing orbit property, $\varphi: X \to \mathbb R^d$ be continuous such that $\mathcal V_\varphi$ is non-trivial.
If $\Delta \subset \mathcal V_\varphi$ is a non-trivial connected set then $X_\Delta$ is Baire residual in $X$.
\end{proposition}

\begin{remark}
The set $\mathcal V_\varphi$ corresponds to the pointwise rotation set of $\varphi$, which needs not be connected in general.
Since Baire residual subsets are preserved by finite intersection, a simple argument by contradiction ensures that under the assumptions of Proposition~\ref{prop:irreg} the set $\mathcal V_\varphi$ is connected. In particular, the latter ensures that
$
X_{\varphi, f}^{wild}
	:=\{x\in X \colon \mathcal V_\varphi(x)
	= \mathcal V_\varphi \}
$
is Baire residual.
\end{remark}

The remaining of the subsection is devoted to the proof of Proposition~\ref{prop:irreg}. Let $D \subset X$ be countable and dense. Let $\vep > 0$ be arbitrary and fixed and let ~$m(\vep)$ be given by the gluing orbit property (cf. Subsection~\ref{gop}).
As $\mathcal V_\varphi\neq \emptyset$
then $\mathcal V_\varphi$ is not a singleton. Let $\Delta\subset \mathcal V_\varphi$ be a non-trivial connected set and, for any $k\ge 1$ let $(v_{k,i})_{1\le i \le a_k}$ be a
$\frac{1}{k}$-dense set of vectors in $\Delta$ so that
\begin{equation}\label{eq:densevec}
\| v_{k,i+1} - v_{k,i} \| < \frac{1}{k}
	\;\text{ for } 1 \le i \le a_k
	\quad\text{and}\quad
	\| v_{k,a_k} - v_{k+1,1} \| < \frac{1}{k}. 	
\end{equation}

For $w\in\mathcal V_\varphi$, $\delta>0$, $n\in\mathbb N$ set
\begin{eqnarray*}\label{def-p}
P(w, \delta, n) = \Big\{x\in~X :  \big\|\frac{1}{n}\sum_{i=0}^{n - 1}\varphi(f^i (x)) - w\big\| < \delta\Big\}.
\end{eqnarray*}

Given $k\ge 1$ and $1\le i \le a_k$, let $\{\delta_{k, i}\}_{k\geq 1,1\le i \le a_k}$ be a sequence of positive real  and $\{n_{k,i}\}_{k\geq 1,1\le i \le a_k}$ be a  sequence of integers tending to zero and infinity, respectively, so that
\begin{align*}\delta_{1,1}^{} > \delta_{1,2}^{} > \dots > \delta_{1,a_1}^{} > \delta_{2,1}^{} > \delta_{2,2}^{} > \dots > \delta_{2,a_2}^{} > \dots, \\
n_{1,1}^{} < n_{1,2}^{} < \dots < n_{1,a_1^{}}^{} < n_{2,1}^{} < n_{2,2} < \dots < n_{2,a_2} < \dots
\end{align*}
that $n_{k,i}^{}\gg m_k$, meaning here $\lim_{k\to\infty} \frac{m_k}{n_{k,i}}=0$, for all $1\le i \le a_k$, where $m_k:=m(\vep/2^{b_{k}})$,  with $b_0=0$ and $b_k = \sum_{1\leq i\leq k} a_i$, and $P(v_{k,i}, \delta_{k,i}, n_{k,i})\neq\emptyset$, for all $k\geq 1$ and $1 \le   i \le  a_k$. Note that $b_k\to\infty $ as
$k\to\infty$.

Given $q\in D$, $k\geq 1$ and $1\le i \le a_k$, let $W_{k,i}(q)$ be a maximal $(n_{k,i}, 8\vep)$-separated subset of $P(v_{k,i}, \delta_{k,i}, n_{k,i})$. We index the elements of $W_{k,i}(q)$ by $x_j^{k,i}$, for $1\le j \le \#W_{k,i}(q)$. Choose also a strictly increasing sequence of integers  $\{N_{k, i}^q\}_{k\geq 1, 1\leq i \leq a_k}$ so that
\begin{align}\label{limdosN9}
& \lim_{k \to~\infty}\frac{n_{k,i+1}^{} + m_k}{N_{k,i}^{q}}=0, \quad \text{for every } 1\leq i \leq a_k-1\nonumber \\
& \lim_{k \to~\infty}\frac{n_{k+1,1}^{} + m_{k+1}}{N_{k,a_k}^{q}}=0, \quad \text{for } k\geq 1  \\
& \lim_{k \to~\infty}\frac{N_{1,1}^q (n_{1,1}^{}+m_1)+ \dots + N_{k,i}^q (n_{k,i}^{}+m_k)}{N_{k,i+1}^{q}}=0
\quad \text{for every },~ 1\leq i \leq a_k -1,\quad  \text{and }\nonumber\\
& \lim_{k \to~\infty}\frac{N_{1,1}^q (n_{1,1}^{}+m_1)+ \dots + N_{k,a_k}^q (n_{k,a_k}^{}+m_k)}{N_{k+1,i}^{q}}=0
\quad \text{for every } ,~k\geq 1\nonumber
\end{align}
We shall omit the dependence of $W_{k,i}(q)$ and $N_{k,i}^q$ on $q$ when no confusion is possible. The idea is to construct points that shadow finite pieces of orbits associated with the vectors $v_{k,i}$ repeatedly.

We need the following auxiliary construction.
The gluing orbit property ensures that for every
$
\underline{x_{k,i}}:=(x_{1}^{k,i}, \dots , x_{N_{k,i}^{}}^{k,i})\in (W_{k,i})^{N_{k,i}}
$
there exists a point $y=y(\underline{x_{k,i}}) \in  X$ and transition time functions
$$
p_{k,i}^{j} : W_{k,i}^{N_{k,i}}\times \mathbb R_{+} \to \mathbb N,
	\qquad j=1, 2, \dots , N_{k,i}-1
$$
bounded above by $m(\frac{\vep}{2^{b_{k-1}+i}}) \le m_k$ so that
\begin{eqnarray}\label{defbj0}
d_{n_{k,i}^{}}(f^{\text{e}_j^{}}(y), x_j^{k,i})<\frac{\vep}{2^{b_{k-1}+i}}, \quad\text{for~every }~j=1, 2, \dots , N_{k,i}-1,
\end{eqnarray}
where
\begin{equation*}\label{defbj}
\text{e}_j^{} =\begin{cases}
0 & \text{if $j=1$}\\
(j - 1)\,n_{k,i}^{} + \sum_{r=1}^{j - 1} p_{k,i}^{r} & \text{if $j= 2, \dots , N_{k,i}$}
\end{cases}.
\end{equation*}

\begin{remark}\label{rmk:differencesG}
For $k$ and $j$ as above we have that
$
p_{k,i}^{j}=p_{k,i}^{j}(x_1^{k,i}, x_2^{k,i}, \dots, x_{N_{k,i}}^{k,i}, \vep)
$
is a function that describes the time lag that the orbit of $y=y({\underline{x_{k,i}}})$ takes to jump from a $\frac\vep{2^k}$-neighborhood of $f^{n_{k,i}^{}}(x_{j}^{k,i})$ to a $\frac\vep{2^k}$-neighborhood of $x_{j+1}^{k,i}$, and it is bounded above by $m_k$. In contrast with the case when $f$ has the specification property, the previous functions need not be constant and, consequently, the collection of points of the form $y=y(\underline{x_{k,i}})$ need not be $3\vep$-separated by a suitable iterate of the dynamics. For that reason, not only an argument to select a `large set' of distinguishable orbits would require to compare points with the same transition times, which strongly differs from \cite{BLV,LW}.
\end{remark}

We order the family $\{W_{k,i}\}_{k\ge 1, 1\le i  \le a_k}$
lexicographically: $W_{k,i} \prec W_{s,j}$ if and only if $k \le s$ and $i\le j$ whenever $k=s$.
We proceed to make a recursive construction of points in a neighborhood of $q$ that shadow points $N_{k,i}$ in the family $W_{k,i}$ successively with bounded  time lags in between.  More precisely, we construct a family $\{L_{k,i}(q)\}_{k \geq 0, 1\le i \le a_k^{}}$ of sets
(guaranteed by the gluing orbit property) contained in a neighborhood of $q$ and a family of positive integers $\{l_{k,i}^{}\}_{k\geq 0, 1\le i \le a_k^{}}$
(also depending on $q$) corresponding to the time during the shadowing process. Set:
\begin{itemize}
\item $L_{0,i}(q) =  \{q\}$ and $l_{0,i}^{} = N_{0,i}^{} n_{0,i}^{} = 0$;

\item $L_{1,1}(q) = \{ z=z(q, y(\underline{x_{1,1}})) \in X  : \underline{x_{1,1}}\in W_{1,1}^{N_{1,1}}\}$ and
$l_{1,1}^{} = p_{1, 1}^0 + t_{1,1}^{}$ with $t_{1,1}^{}= N_{1,1} n_{1,1}^{} + \sum_{r=1}^{N_{1,1} - 1} p_{1,1}^{r}$,
where $z=z(q, y(\underline{x_{1,1}}))$ satisfies $d(z, q) < \frac{\vep}{2}$
and $d_{t_{1,1}^{}}(f^{p_{1, 1}^0}(z), y(\underline{x_{1,1}})) < \frac{\vep}{2}$,
and $y(\underline{x_{1,1}})$ is defined by ~\eqref{defbj0} and $0\le p_{1, 1}^0 \le m(\frac\vep{2^2})$ is given by the gluing orbit property;

\item if  $i=1$\\
$L_{k,1}(q) = \{ z=z(z_0, y(\underline{x_{k,1}})) \in X  \colon \underline{x_{k,1}} \in  W_{k,1}^{N_{k,1}}
\text{ and } z_0 \in L_{k-1,a_{k-1}} \}$,
and $l_{k,1} = l_{k-1,a_{k-1}} + p_{k,1}^{0} + t_{k,1}$, with
$t_{k,1}= N_{k,1} n_{k,1}^{} + \sum_{r=1}^{N_{k,1} - 1} p_{k,1}^{r}$,
where the shadowing point $z$ satisfies
$$
\qquad d_{l_{k-1,a_{k-1}}}(z, z_0) < \frac{\vep}{2^{b_{k-1} + 1}} \;\text{ and }\;
d_{t_{k,1}}(f^{l_{k-1,a_{k-1}}^{} + p_{k,1}^{0}}(z), y(\underline{x_{k,1}})) < \frac{\vep}{2^{b_{k-1} + 1}}.
$$

\item if  $i~\neq~1$\\
$L_{k,i}(q) = \{ z=z(z_0, y(\underline{x_{k,i}})) \in X  \colon \underline{x_{k,i}} \in  W_{k,i}^{N_{k,i}}
\text{ and }  z_0 \in L_{k,i-1} \}$,
and $l_{k,i} = l_{k,i-1} + p_{k,i}^{0} + t_{k,i}$, with
$t_{k,i}= N_{k,i} n_{k,i}^{} + \sum_{r=1}^{N_{k,i} - 1} p_{k,i}^{r}$,
where the shadowing point $z$ satisfies
$$
d_{l_{k,i-1}}(z,z_0) < \frac{\vep}{2^{b_{k-1} + i}}~\text{and}~
d_{t_{k,i}}(f^{l_{k,i-1}^{} + p_{k,i}^{0}}(z), y(\underline{x_{k,i}})) < \frac{\vep}{2^{b_{k-1} + i}}.
$$
\end{itemize}

The previous points $y=y(\underline{x_{k,i}})$ are defined as in ~\eqref{defbj0}.
By construction, for every $k\ge 1$ and $1\le i \le a_k-1$,
\begin{equation}\label{deflk}
l_{k,i}^{}=\sum_{r=1}^{k}\sum_{s=1}^{a_r} N_{r,s} n_{r,s}^{} + \sum_{r=1}^{k} \sum_{s=1}^{a_r^{} - 1} \sum_{t=0}^{N_{r,s}-1}  p_{r,s}^t.
\end{equation}

\begin{remark} \label{oblimitali}
Note that $l_{k,i}^{}$ and $t_{k,i}^{}$ are functions (as these depend on $p_{k, i}^j$)  and, by definition of $N_{k,i}$ cf. \eqref{limdosN9}, one has that
$
\frac{\|l_{k,i}^{}\|}{N_{k,i+1}} \leq \frac{ \sum_{r=1}^{k} \sum_{s=1}^{a_r^{}} N_{r,s}( n_{r,s}^{} + m_r^{})}{N_{k,i+1}}
$
tends to zero as $k \to\infty$.
\end{remark}

For every $k\geq 0$, $1\le i \le a_k$, $q\in D$ and $\vep>0$ define
\begin{eqnarray*}
R_k (q,\vep, i) = \bigcup_{z\in L_{k,i}(q)} \widetilde{B}_{l_{k,i}^{}} \Big(z, \frac\vep{2^{b_{k-1}+i}}\Big)
	\quad\text{and}\quad
	R(q,\vep) = \bigcap_{k = 0}^{\infty} \bigcap_{i = 1}^{a_k^{}} R_k (q,\vep, i),
\end{eqnarray*}
where $\widetilde{B}_{l_{k,i}} (x, \delta)$ is the set of points $y \!\in\! X$ so that $d(f^\alpha (x), f^\alpha (y)) < \delta$ for all iterates
$
0 \le \!\alpha\! \le l_{k,i-1}^{}-1$ and $d(f^\beta (x), f^\beta (y)) \leq \delta \, \text{for every}\,  l_{k,i-1}^{}\le \beta \le l_{k,i}^{} - 1.
$

Consider also the sets
\begin{equation}
 \widetilde{\mathcal R}= \bigcup_{j = 1}^{\infty} \bigcup_{q\in D} R(q, \frac1{j})
 	=\bigcup_{j = 1}^{\infty} \bigcup_{q\in D}\bigcap_{k = 0}^{\infty} \bigcap_{i = 1}^{a_k^{}}
	\bigcup_{z\in L_{k,i}(q)} \widetilde{B}_{l_{k,i}^{}}\Big(z, \frac{1}{j2^{b_{k-1}+i}}\Big),
\end{equation}
and finally
\begin{equation}\label{eq:residual}
\mathcal R=  \bigcap_{k = 0}^{\infty}\bigcup_{j = 1}^{\infty} \bigcup_{q\in D} \bigcap_{i = 1}^{a_k^{}}
\bigcup_{z\in L_{k,i}(q)}  \widetilde{B}_{l_{k,i}^{}}\Big(z, \frac{1}{j2^{b_{k-1} + i}}\Big).
\end{equation}
It is clear from the construction that $R(q,\vep)\subset B(q,\vep)$ for every $q\in D$ and $\vep>0$, and that $\tilde {\mathcal   R} \subset \mathcal R$. The following lemma, identical to Propositions 2.2 and 2.3 in ~\cite{LW}, ensures that $\mathcal R$ is a Baire generic subset of $X$.

\begin{lemma}\label{gdeltadenso}
$\mathcal R$ is a $G_\delta$-set and it is dense in $X$.
\end{lemma}

\begin{proof}
First we prove denseness.
Since
$
\widetilde{\mathcal R} \subset \mathcal R,
$
it is enough to show that $\widetilde{\mathcal R}\cap B(x, r) \neq \emptyset$ for every $x\in X$ and $r > 0$. In fact, given $x\in X$ and $r > 0$, there exists $j\in\mathbb N$ and $q \in D$ such that $d(x, q) < 1/j < r/2$. Given $y\in R(q,\frac1j)$ it holds that $d(q, y)<\frac1j$  because $R(q,\frac1j)\subset B(q,\frac1j)$. Therefore, $d(x, y) \leq d(x, q) + d(q, y) < 2/j < r.$ This ensures that $\widetilde{\mathcal R}\cap B(x, r) \neq \emptyset$.

Now we prove that $\mathcal R$ is a $G_\delta$-set.
Fix $j\in\mathbb N$ and $q\in D$. For any $k \ge 1$ and $1\le i \le a_k$, consider the open set
$$
G_{k}(q,\vep,i):=\bigcup_{z\in L_{k,i}(q)} B_{l_{k,i}^{}}\Big(z, \frac{\vep}{2^{b_{k-1} + i}}\Big)
$$
and note that $G_{k}(q,\vep,i) \subset  R_{k}(q,\vep,i)$ for any $k \ge 1$ and $1\le i \le a_k^{}$.
We claim that $ R_{k}(q,\vep,i+1)\subset  G_{k}(q,\vep,i)$ and  $ R_{k+1}(q,\vep,1)\subset G_{k}(q,\vep,a_k^{})$, for any $k \ge 1$ and $1\le i \le a_k-1$. The claim implies that
$$
\bigcup_{j = 1}^{\infty} \bigcup_{q\in D}
	\bigcap_{i = 1}^{a_k^{}} R_{k}(q,\vep,i)
	= \bigcup_{j = 1}^{\infty} \bigcup_{q\in D}
	\bigcap_{i = 1}^{a_k^{}} G_{k}(q,\vep,i),
$$
and guarantees that $\mathcal R$ is a $G_\delta$-set.

Now we proceed to prove the claim. We  prove that
$ R_{k}(q,\vep,i+1)\subset  G_{k}(q,\vep,i)$ for any $k \geq 1$ and $1\le i \le a_k^{} - 1$ (the proof of the $ R_{k+1}(q,\vep,i)\subset  G_{k}(q,\vep,a_k^{})$ is analogous). Given $y\in  R_{k}(q,\vep,i+1)$, there exists $z\in L_{k,i+1}(q)$ such that $y \in \widetilde{B}_{l_{k,i+1}^{}}(z, \frac{1}{j2^{b_{k-1} + (i+1)}})$. By definition of $L_{k,i+1}(q)$, there exists $z_0$ such that $d_{l_{k,i}^{}}(z,z_0)<\frac1{j2^{b_{k-1} + (i+1)}}$. Therefore,
$$
d_{l_{k,i}^{}}(y,z_0)
\leq d_{l_{k,i}^{}}(y,z) + d_{l_{k,i}^{}}(z,z_0)
< \frac{1}{j2^{b_{k-1} + (i+1)}}+\frac{1}{j2^{b_{k-1} + (i+1)}}
=\frac{1}{j2^{b_{k-1}+i}}
$$
and consequently $y\in  G_{k}(q,\vep,i)$. This proves the claim and completes the proof of the lemma.
\end{proof}

We must show that $\mathcal R \subset X_\Delta$, that is $\Delta \subseteq \mathcal V_\varphi (x)$ for every $x\in \mathcal R$.
The proof follows some ideas from \cite[Proposition~2.1]{LW}. We provide a sketch of the argument for completeness.
Given $x\in \mathcal R$ fixed, for any $k > 1$, there exist integers
$j\in\mathbb N$, $q\in D$ and $z\in L_{k, i+1}(q)$ such that
\begin{eqnarray}\label{est29}
d_{l_{k, i+1}^{}} (z,x)<\frac{1}{j2^{b_{k-1}}} \; ~~\text{for every } 1\le i \le a_k^{}.
\end{eqnarray}

We prove that $ \Delta \subseteq \mathcal V_\varphi (x) $.
If $v\in \Delta$ then for any $k\ge 1$ there exists $1 \le i_k \le a_k$ such that $v\in B(v_{k,i_k}^{}, \frac{1}{k})$. We need the following:

\begin{lemma}\label{lemma21}
Take $k \geq 1$ and $1 \le i_k \le a_k $. If
$$R_{k,i}^{q}:=\max_{z\in L_{k,i}(q)}\Big\|\sum_{r=0}^{l_{k,i}^{} - 1}\varphi(f^{r} (z)) - l_{k,i}^{}\,  v_{k,i}^{}\Big\|$$
then~~$\frac{R_{k,i}^{q}}{l_{k,i}^{}} \to 0, ~~~\text{as}~k\to\infty$.
\end{lemma}

\begin{proof}
Let $k$ and $i$ be as above, let  $\underline{x_{k,i}}\in (W_{k,i})^{N_{k,i}}$ and $y=y(\underline{x_{k,i}})$.
Recall that
$\|\sum_{i=0}^{n - 1}\varphi(f^{i} (x)) - \sum_{i=0}^{n - 1}\varphi(f^{i} (y))\|\leq n\var(\varphi, c)$
if $d_n (x, y) < c$.
Then, using
$d_{n_{k,i}}(x_{t}^{k,i}, f^{e_t}(y))<\frac{1}{j2^{b_{k-1}}}$ where
$e_t$ is defined in \eqref{defbj} with $t\in\{1, \dots, N_{k,i}\}$
we conclude that
$$
\Big\|\sum_{r=0}^{n_{k,i} - 1}\varphi(f^{r} (x_{t}^{k,i})) - \sum_{r=0}^{n_{k,i} - 1}\varphi(f^{e_t+r} (y))\Big\|\leq n_{k,i}\var(\varphi, \frac{1}{j2^{b_{k-1}}}).
$$
Since $x_{t}^{k,i}\in W_{k,i}$, we have that
\begin{eqnarray}\label{est213}
\Big\|\sum_{r=0}^{n_{k,i} - 1}\varphi(f^{e_t+r} (y)) - n_{k,i}\,v_{k,i}\Big\|\leq n_{k,i}(\var(\varphi, \frac{1}{j2^{b_{k-1}}})+\delta_{k,i}).
\end{eqnarray}
We decompose the time interval $[0, t_{k,i} - 1]$ as follows:
$$
\bigcup_{t~=~1}^{N_{k,i}}[e_{t}, e_{t} + n_{k,i} - 1]
\cup
\bigcup_{t~=~1}^{N_{k,i}-1}[e_{t} + n_{k,i}, e_{t} + n_{k,i} + p_{k,i}^{t} - 1].
$$
On the intervals $[e_t, e_t + n_{k,i} - 1]$ we will use the estimate \eqref{est213}, while in the time intervals
$[e_{t} + n_{k,i}, e_{t} + n_{k,i} + p_{k,i}^{t} - 1]$ we use
$$
\Big\|\sum_{r=0}^{p_{k,i}^{t} - 1}\varphi(f^{e_t+n_{k,i}+r} (y)) - p_{k,i}^{t}\,v_{k,i}\Big\|
\leq m_{k}(\|\varphi\|_{\infty} + \|v_{k,i}\|)\leq 2 \,m_k \|\varphi\|_{\infty}.
$$
Therefore,
\begin{eqnarray}\label{est214}
\Big\|\sum_{r=0}^{t_{k,i} - 1}\varphi(f^{r} (y)) - t_{k,i}\,v_{k,i}\Big\|
\leq N_{k,i}\,n_{k,i}^{} (\var(\varphi, \frac{1}{j2^{b_{k-1}}})+\delta_{k,i})+2(N_{k,i}-1) m_k^{} \|\varphi\|_\infty.
\end{eqnarray}
On the other hand, by definition of $L_{k,i}(q)$ that for every $z\in L_{k,i}(q)$ there exist $z_0\in L_{k,i-1}(q)$ and  $y=y(\underline{x_{k,i}}) \in  X$ such that
\begin{eqnarray}\label{est215}
d_{l_{k,i-1}}(x, z)<\frac{1}{j2^{b_{k-1}}},~~~~~d_{t_{k,i}}(y, f^{l_{k,i-1}+p_{k,i}^0}(z))<\frac{1}{j2^{b_{k-1}}}
\end{eqnarray}
By triangular inequality,
\begin{eqnarray*}
\Big\|\sum_{r=0}^{l_{k,i} - 1}\varphi(f^{r} (z)) - l_{k,i}\,v_{k,i}\Big\|&\leq& \Big\|\sum_{r=0}^{l_{k,i-1} - 1}\varphi(f^{r} (z)) - l_{k,i-1}\,v_{k,i}\Big\|\\
&+&\Big\|\sum_{r= p_{k,i}^0}^{l_{k,i-1} + p_{k,i}^0 - 1}\varphi(f^{r} (z)) - p_{k,i}^0\,v_{k,i}\Big\| \\
&+&\Big\|\sum_{r=l_{k,i-1}+ p_{k,i}^0}^{t_{k,i} - 1}\varphi(f^{r} (z)) - t_{k,i}\,v_{k,i}\Big\|,
\end{eqnarray*}
where the first and second terms are bounded by $2 l_{k,i-1}\|\varphi\|_\infty$ and $2 m_{k}\|\varphi\|_\infty$, respectively.
Inequalities \eqref{est214}-\eqref{est215} imply
\begin{eqnarray*}
&&\Big\|\sum_{r=t_{k,i-1}+ p_{k,i}^0}^{t_{k,i} - 1}\varphi(f^{r} (z)) - t_{k,i}\,v_{k,i}\Big\|\\
&\leq&\Big\|\sum_{r=0}^{t_{k,i} - 1}\varphi(f^{l_{k,i-1} + p_{k,i}^0 + r} (z)) - \sum_{r=0}^{t_{k,i} - 1}\varphi(f^{r} (y))\Big\|+\Big\|\sum_{r=0}^{t_{k,i} - 1}\varphi(f^{r} (y)) - t_{k,i}\,v_{k,i}\Big\|\\
&\leq& t_{k,i}\var(\varphi, \frac{1}{j2^{b_{k-1}}}) + N_{k,i} n_{k,i}^{}\var(\varphi, \frac{1}
{j2^{b_{k-1}}}+\delta_{k,i}) +  2 (N_{k,i} -1)m_{k}\|\varphi\|_\infty
\end{eqnarray*}
and, consequently,
\begin{eqnarray*}\label{est216}
R_{k,i}^q \leq 2 (l_{k,i-1} + N_{k,i} m_{k} )\|\varphi\|_\infty + (t_{k,i} + N_{k,i} n_{k,i})\var(\varphi, \frac{1}{j2^{b_{k-1}}}) + N_{k,i} n_{k,i}^{} \delta_{k,i}.
\end{eqnarray*}
By definition of $N_{k,i}$ in  \eqref{limdosN9} we obtain that $R_{k,i}^q/l_{k,i}\to 0$ as $k\to\infty$, which proves the lemma.
\end{proof}

Given $z = z(z_0, y(\underline{x_{k,i}})) \; \in  L_{k,i+1}^{}(q)$ satisfying \eqref{est29} with $z_0\in L_{k,i}^{}(q)$,
by triangular inequality we have $d_{l_{k,i}}(z_0, x)<\frac{1}{j2^{b_{k-1}-1}}$. Thus,
\begin{eqnarray}\label{est217}
&&\Big\|\sum_{r=0}^{l_{k,i_k} - 1}\varphi(f^{r} (x)) -  l_{k,i_k^{}}\,v_{k,i_k^{}}\Big\|\\
&\leq& \Big\|\sum_{r=0}^{l_{k,i_k}^{} - 1}\varphi(f^{r} (x)) - \sum_{r=0}^{l_{k,i_k^{}}^{} - 1}\varphi(f^{r} (z_0))\Big\| + \Big\|\sum_{r=0}^{l_{k,i_k} - 1}\varphi(f^{r} (z_0)) -  l_{k,i_k^{}}^{}\,v_{k,i_k^{}}\Big\|\nonumber\\
&\leq& l_{k,i_k^{}}^{} \var(\varphi, \frac{1}{j2^{b_{k-1}-1}}) + R_{k,i_k^{}}^q.\nonumber
\end{eqnarray}
Lemma~\ref{lemma21} and the uniform continuity of $\varphi$ ensures that
\begin{eqnarray}\label{snconva}
\Big\|\frac{1}{l_{k,i}^{}}\sum_{r=0}^{l_{k,i}^{} - 1}\varphi(f^{r} (x)) - v\Big\|\leq \Big\|\frac{1}{l_{k,i}^{}}\sum_{r=0}^{l_{k,i}^{} - 1}\varphi(f^{r} (x)) - v_{k,i_k}^{}\Big\| + \|v_{k,i_k}^{} - v\| \to 0
\end{eqnarray}
as $k\to\infty$ and, consequently, $v\in \mathcal V_\varphi (x)$. This proves that $\Delta \subseteq \mathcal V_\varphi (x)$.

Altogether we conclude that $X_\Delta$ is a Baire residual subset of $X$, and finish
the proof of Proposition \ref{prop:irreg} and Theorem~\ref{thm:A}.

\subsection{Full topological pressure and metric mean dimension}\label{subsec:htopmmd}

In this section we prove Theorem \ref{thm:B}. Assume that $f$ is a continuous map with the gluing orbit property on a compact metric
space $X$ and that $\varphi : X \rightarrow \mathbb R^d$ is continuous such that $X_{\varphi, f}\neq\emptyset$.  The proofs of
(i) $\displaystyle  \underline{\text{mdim}}_{X_{\varphi, f}}^{} (f) = \underline{\text{mdim} (f)}$, (ii)
$\displaystyle  \overline{\text{mdim}}_{X_{\varphi, f}}^{} (f) = \overline{\text{mdim} (f)}$ and (iii) 
$h_{top}^{}(f) = h_{X_{\varphi, f}^{}}^{}(f)$ will be a consequence from the fact that
$h_{X_{\varphi, f}^{}}^{}(f,\vep)=\htop(f,\vep)$ for every $\vep>0$. Fix $\vep>0$ and let $m(\vep)$ be given by the gluing orbit property.

\subsubsection{Measures with large entropy and distinct rotation vectors }
The proof explores the construction of an exponentially large (with exponential rate close to topological entropy) number of points
that oscillate between distinct vectors in $\mathbb R^d$.
We use some auxiliary results. We say that a observable $ \varphi: X \to \mathbb R^d$ is  \emph{cohomologous to a vector}  if there exists $v \in \mathbb R^d$ and
a continuous function $\chi: X \to \mathbb R^d$ so that  $\varphi =v +  \chi-   \chi\circ f$, and denote by $Cob$ the set of all such
observables and by $\overline{Cob}$ its closure in the $C^{0}$-topology.
\begin{lemma}\label{obsat}
Assume that $f$ has the gluing orbit property. The following are equivalent:
\begin{itemize}
\item[(i)] $X_{\varphi,f} \neq \emptyset$;
\item[(ii)] there are $\mu_1,\mu_2\in \mathcal M_e(f)$ such that $ \int \varphi\, d\mu_1\neq \int \varphi\, d\mu_2$; 
\item[(iii)] there exist periodic points $p_1,p_2$ of period $k_1, k_2$ respectively such that
	$$\frac1{k_1} \sum_{j=0}^{k_1-1} \varphi(f^j(p_1)) \neq \frac1{k_2} \sum_{j=0}^{k_2-1} \varphi(f^j(p_2));$$
\item[(iv)] $\varphi \notin  \overline{Cob}$;
\item[(v)] $\frac1n \sum_{j=0}^{n-1} \varphi \circ f^j $ does not converge uniformly to a constant.
\end{itemize}
\end{lemma}

\begin{proof}
Although this is similar to \cite[Lemma~1.9]{TMP} we include it for completeness.

$(iii)\Rightarrow (iv)$:
If $ \varphi\in \overline{Cob}$, then there is $ \{\varphi_k\}$ in~$Cob$ such that
$ \varphi = \lim_{k \to \infty} \varphi_k$. In particular there exists $v_k\in \mathbb R^d$  and $\chi_k$ continuous  so that
\begin{equation}\label{eq:cohom}
\frac1n \sum_{j=0}^{n-1}\varphi_k(f^j(x))= \frac{\chi_k \circ f^{n}(x)}{n} - \frac{\chi_k(x)}{n} +v_k
\end{equation}
for every $x\in X$. By Birkhoff's ergodic theorem and dominated convergence theorem  $\mathcal M_e(f) \ni\mu\mapsto \int \varphi d\mu$ is constant, which contradicts (iii).

$(iv)\Rightarrow (v)$:
If $ \varphi\not\in \overline{Cob}$ then the sequence $\frac1n \sum_{j=0}^{n-1} \varphi\circ f^j$ is not
uniformly convergent to a vector $v$. Indeed, otherwise
the sequence of continuous function $(h_n)_n$ given by $ h_n = \frac{1}{n}\sum_{i=0}^{n -1} \, (n-i)\varphi\circ f^{i-1}$
satisfy the cohomological equation
$$
h_n^{} (x) -h_n^{} (f(x)) = \varphi(x) - \frac1n \sum_{j=0}^{n-1} \varphi(f^j(x)), ~~\forall~x\in X
$$
and so
$
\varphi(x) = \lim_{n\to\infty} [h_n^{} (x) -h_n^{} (f(x)) +  \frac1n \sum_{j=0}^{n-1} \varphi(f^j(x))] \in \overline{Cob},
$
leading to a contradiction.

$(v)\Rightarrow (ii)$:
Let $\mu$ be an $f$-invariant probability measure and suppose that $\frac1n \sum_{j=0}^{n-1} \varphi(f^j(x))$ does not converge uniformly to $\int \varphi d\mu$. There exists $\vep > 0$ so that for every $k\geq 1$ there are $n_k\geq k$ and $x_k\in X$ for which
$\|\frac1{n_k^{}} \sum_{j=0}^{n_k^{}-1} \varphi(f^j(x_k)) - \int \varphi d\mu\|\geq \vep$. Consider $\nu_k := \frac{1}{n_k}\sum_{j = 0}^{n_k - 1} \delta_{f^j (x_k)}$ and let $\nu$ be a weak$^*$ accumulation point of the sequence $(\nu_k)_k$. 
Note that $\nu$ is $f$-invariant. Choose $k$ such that $\|\frac1{n_k} \sum_{j=0}^{{n_k}-1} \varphi(f^j(x_{k})) - \int \varphi d\nu\|\leq \vep/2$,  so
\begin{eqnarray*}
\big\|\int \varphi d\mu - \int \varphi d\nu\big\| &\geq&\big\|\int \varphi d\mu - \frac1{n_k} \sum_{j=0}^{{n_k}-1} \varphi(f^j(x_{k}))\big\|\\
&-&\big\|\frac1{n_k} \sum_{j=0}^{{n_k}-1} \varphi(f^j(x_{k})) - \int \varphi d\nu\big\|\geq \vep/2.
\end{eqnarray*}
The conclusion follows from the ergodic decomposition theorem.

$(ii)\Rightarrow (i)$:
The construction in the proof of Theorem~\ref{thm:A} ensures that if there exist $f$-invariant measures $\mu_1,\mu_2$
so that $\int \varphi ~d \mu_1 \neq \int \varphi ~d \mu_2 $ then
 $X_{\varphi, f}\neq\emptyset$.

$(i)\Rightarrow (iii)$:
If the limit $\lim_{n\to\infty} \frac1n \sum_{j=0}^{n-1} \varphi(f^j(x))$ does not exist for some $x\in X$
then the empirical measures $( \frac1n \sum_{j=0}^{n-1} \delta_{f^j(x)})_{n\ge 1}$ accumulate on $f$-invariant probability measures $\mu_1, \mu_2$ so that $\int \varphi\, d\mu_1 \neq \int \varphi\, d\mu_2$. Now, the result follows as a simple consequence of the weak$^*$ convergence and the fact that periodic measures are dense in the space of $f$-invariant probability measures 
(cf. Lemma~\ref{lemma:per-measures}).
\end{proof}

\begin{lemma}\label{exist-medmu}
Given $\psi\in C^0(X,\mathbb R)$and  $\gamma > 0$ there are  $\mu_1,\mu_2^{}\in \mathcal  M_1(f)$ so that $\mu_1$ is ergodic, $\int \varphi d\mu_1^{} \neq \int \varphi d\mu_2^{}$ and $h_{\mu_i}^{} (f) + \int \psi d\mu_i> P_{top}^{} (f, \psi)- \gamma$,  for $i=1, 2$.
\end{lemma}

\begin{proof}
By the variational principle  there exists an ergodic $\mu_1^{}\in  \mathcal M_1(f) $ so that  $h_{\mu_1}^{} (f) + \int \psi d\mu_1> P_{top}^{} (f, \psi)- \gamma$. As $X_{\varphi, f}^{}\neq\emptyset$
there is $\nu\in \mathcal M_1(f)$ satisfying  $\int \varphi d\mu_1^{} \neq \int \varphi d\nu$ (recall Lemma~\ref{obsat}). Consider the family of measures
\begin{eqnarray}\label{defdemu2}
\mu_2^{t} = t\mu_1^{} + (1-t)\nu, \quad t\in (0, 1)
\end{eqnarray}
and observe that, by convexity,
$
h_{\mu_2^t}^{}(f)  + \int \psi d\mu_2^t
	> P_{top}^{} (f, \psi)- \gamma,
$
provided that the constant $t=t(\gamma,\psi)\in (0,1)$ is sufficiently close to one. 
Note that $t\to 1$ as $\gamma \to 0$ and that the probability measure $\mu_2:=\mu_2^t$ satisfies the requirements of the lemma.
\end{proof}

Although the previously defined measures $\mu_1, \mu_2$ depend on the potential $\psi \in C^0(X,\mathbb R)$ 
and $t\in (0,1)$ close to one, we shall omit its dependence
for notational simplicity when possible.

\subsubsection{Exponential growth of points with averages close to $\int \varphi\, d\mu_i$, $i=1,2$}

Take $\gamma\in (0,1)$ arbitrary and take $t\in (0,1)$ and the probability measures $\mu_1^{}, \nu$ and $\mu_2$ given by Lemma \ref{exist-medmu}. Consider the sequence $\{\zeta_k^{}\}_k$ of real numbers
\begin{eqnarray} \label{defzetak}
\zeta_k^{} = \max \Big\{\frac{\|\int \varphi d\mu_1^{} - \int \varphi d\nu\|}{2^k}, \var(\varphi, \frac\vep{2^k})\Big\},
\end{eqnarray}
which tend to zero as $k\to\infty$, and take $m_k=m(\frac{\vep}{2^k})$. By Birkhoff's ergodic theorem one can choose $n_k^1 \gg m_k$ so that $\mu_1(Y_{k, 1}) \ge 1-\gamma$, where
\begin{eqnarray}
Y_{k, 1}= \Big\{x\in X : \Big\|\frac{1}{n}\sum_{j = 0}^{n-1}\varphi(f^{j}(x))- \int \varphi d\mu_{1}\Big\|<\zeta_k^{} \; \text{ for every } n\ge n_k^1\Big\}.
\end{eqnarray}
We make the previous choice in such a way that 
\begin{equation}\label{eqmko}
m_k / n_k^1 \to 0 \quad \text{as $k\to\infty$}.
\end{equation}
The following lemma will be instrumental.

\begin{lemma}\label{constnkesk} There exists $\vep_0^{} > 0$ so that for any $0< \vep < \vep_0^{} $, there is a collection $\{S_k^1\}_k$  so that every $S_k^1$ is a $(n_k^1, 6\vep)$ separated subset of
$Y_{k, 1}$ and $M_k^1:=\sum_{x\in S_k^1}  \exp (\sum_{i=0}^{n^1_k -1} \psi(f^i (x)))$ satisfies $M_k^1\geq \exp\; (n_k^1(P_{top}^{}(f,\psi,\vep) - 4\gamma))$.
\end{lemma}
\begin{proof}
The proof is a standard consequence of Proposition~\ref{TMP}.
\end{proof}

For any $k\ge 1$, we now construct large sets of points $S_k^2$ with time averages close to $\int\varphi\, d\mu_2$ at large instants $n_k^2 \ge 1$ (to be defined below). First, as $\nu$ is ergodic, there exists $\widetilde{x_{k}^{}}\in X $ and $\tilde n_k\ge 1$ so that
$
\big\|\frac{1}{\widetilde{n}_k }\sum_{j = 0}^{\widetilde{n}_k -1}\varphi(f^{j}(\widetilde{x_{k}^{}}))- \int \varphi d\nu\big\|<\zeta_k^{}.
$
Choose two sequences of integers $r_k^{},s_k^{} \ge 1$ satisfying
\begin{eqnarray}\label{relrknk}
~~\frac{r_k^{} n_k^1}{s_k^{} \widetilde{n}_k^{}} \to~\frac{t}{1-t} \qquad \text{as}~~ k\to \infty,
\end{eqnarray}
where $t\in (0,1)$ is as above.

For any fixed $k\geq 1$, any string $(x_1^k, \dots , x_{r_k}^{k})\in (S_k^1)^{r_k^{}}$ and $s_k^{}$ copies of the point $\widetilde{x_k}$, by the gluing orbit property there exists $y=y(x_1^k, \dots , x_{r_k}^{k}) \in X$ satisfying
\begin{eqnarray*}
d_{n_k^{1}}  (f^{a_i^{}}(y), x_{i}^{k} ) < \vep,
	\qquad\text{and}\qquad
	d_{\widetilde{n}_k^{}} (f^{b_j^{}}(y), \widetilde{x_{k}} )<\vep
\end{eqnarray*}
for every $i=1, 2, \dots , r_k^{}$ and $j=1, 2, \dots , s_k^{}$, where
\begin{equation*}
a_i^{} =\begin{cases}
0 & \text{if $i=1$}\\
(i - 1)n_{k}^{1}  + \sum_{r=1}^{i-1} p_{k,r}^{} & \text{if $i= 2, \dots , r_k^{} $}
\end{cases}
\end{equation*}
\noindent and
\begin{equation*}
b_j^{} =\begin{cases}
 a_{r_k^{}}^{} + p_{k,r_k^{}}^{} & \text{if $j=1$}\\
(j - 1)\widetilde{n}_{k}^{} + \sum_{r=0}^{j-1} p_{k,r_k^{}+r}^{}+a_{r_k^{}} & \text{if $j= 2, \dots , s_k^{} $}
\end{cases}
\end{equation*}
where $0\leq p_{k,r}^{}\leq m(\vep)$ are the transition time functions defined similarly as in the proof of Theorem \ref{thm:A}. We define the auxiliary set $\widehat{S_k^2}$ as the set of points $y$ obtained by the previous process.

\begin{remark}
For every point $x\in~\widehat{S_k^2}$ we associate the size
$$
n_k^2 (\cdot) := r_k^{} n_k^{1} + s_k^{} \widetilde{n}_k^{} +  \sum_{r=1}^{r_k^{}+s_k^{}-1} p_{k,r}^{}(\cdot),
$$
of the finite piece of orbit,  which is a function of $(x_1^k, \dots , x_{r_k}^k, \widetilde{x_k}, s_k)$. In strong contrast with the case when $f$ satisfies the specification property, at this moment we can not claim that the cardinality of $\widehat{S_k^2}$ is large. Indeed, since  $n_k^2$ varies with the elements in $\widehat{S_k^2}$ then the $(n_k^1,4\vep)$-separability of the points in $S_k^1$ is not sufficient to ensure the shadowing point map $(S_k^1)^{r_k} \times \{\widetilde{ x_k}\}^{s_k} \to X$ to be injective. This issue is solved by Lemma~\ref{lemaestimSK2}.
\end{remark}

Now, for any $\underline{j}=(j_1, j_2, \dots , j_{r_k+s_k-1})\in \mathbb Z_{+}^{r_k+s_k-1}$ so that
$0\leq j_i\leq m(\vep)+1$  define the set
$
S_k^2(\underline{j}):=\{x\in\widehat{S_k^2}~: ~p_{k, 1}=j_1,~p_{k, 2}=j_2, \dots , p_{k, r_k+s_k-1}=j_{r_k+s_k-1}\}.
$
The size of the finite orbit of all points in $S_k^2 (\underline{j^{}})$ is constant and, by some abuse of notation,
we will denote it by
\begin{eqnarray}\label{defdenk2}
n_k^2  (\underline{j^{}}) := r_k^{} n_k^{1} + s_k^{} \widetilde{n}_k^{} +  \sum_{r=1}^{r_k^{}+s_k^{}-1} j_{r}^{}.
\end{eqnarray}
It is not hard to check that ~\eqref{relrknk} implies
\begin{eqnarray}\label{zeronk2}
\frac{r_k^{} n_k^1}{r_k^{} n_k^1+s_k^{} \widetilde{n}_k^{}}
	\to t\quad\text{and, consequently, }\quad
	\frac{r_k^{} n_k^1}{n_k^2  (\underline{j^{}})} \to t
\end{eqnarray}
as $k\to\infty$.  Moreover,
\begin{eqnarray}\label{skrk}
\frac{r_k^{} +s_k}{n_k^2} \le \frac{1}{n_k^1} + \frac{1}{\widetilde{n}_k}
	\to 0 \quad \text{as $k\to\infty$.}
\end{eqnarray}
The next lemma says that one can choose a large set $S_k^2$ of points whose $n_k^2$-time average is close to the one determined
by $\mu_2$. More precisely:

\begin{lemma} \label{lemaestimSK2}
For every large $k\geq 1$ there exists
$
\underline{j_k^{}}= (j_1^{k}, \dots , j_{r_k^{k}+s_k^{}-1}^{k})$  so that if $S_k^2:=S_k^2 (\underline{j_k^{}})
$
and
$n_k^2=n_k^2  (\underline{j_k^{}})$ then the following hold:
\begin{enumerate}
\item $S_k^2$ is $(n_k^2, 4\vep)$-separated,
\item  if $M_k^2:=\sum_{x\in S_k^2} \exp \{\sum_{i=0}^{n_k^2 -1} \psi(f^i (x))\}$ then
    $$
    M_k^2 \geq \exp\; (n_k^2 [ t P_{top}^{}(f,\psi,\vep) - \var(\psi,\vep) - 6\gamma ] ),
    $$
	\item there exists a sequence $(a_k)_{k\ge 1}$ converging to zero so that
	$$
	 \big\|\frac{1}{n_k^2}\sum_{j = 0}^{n_k^2-1}\varphi(f^{j}(y))- \int \varphi d\mu_2\big\|
	\le \var(\varphi, \vep) + a_k \|\varphi\|_\infty,
	$$
	for every $k\ge 1$ and every $y\in S_k^2$.
\end{enumerate}
\end{lemma}

\begin{proof}
In order to prove item (1), let
$\underline{j}=(j_1, j_2, \dots , j_{r_k+s_k-1})\in \mathbb Z_{+}^{r_k+s_k-1}$ be arbitrary so that  $0\leq j_i\leq m(\vep)+1$.
Let $y_1^{}\neq y_2^{}\in S_k^2(\underline j)$ shadow the orbits of points in the strings $(x_1^{k}, \dots , x_{r_k^{}}^{k}) \neq (z_1^{k}, \dots , z_{r_k^{}}^{k}) \in (S_k^1)^{r_k}$ and also $s_k$ times the finite piece of orbit of $\widetilde{x_k}$, respectively.
There exists $1\le i \le r_k$ such that $x_i^{} \neq z_i^{}$ and, using that $S_k^1$ is $(n_k^1,6\vep)$-separated,
\begin{eqnarray*}
d_{{n_k^{2}(\underline j)}}^{} (y_{1}, y_{2})
	\geq  d_{n_k^{1}}(x_i^{k}, z_i^{k}) -  d_{n_k^{1}}(y_{1}, x_i^{k}) - d_{n_k^{1}}(z_i^{k}, y_{2})
	\ge 4\vep.
\end{eqnarray*}
Therefore,  $S_k^2(\underline j)$ is $(n_k^2(\underline j), 4\vep)$-separated for every $\underline j$. This implies (1).

\medskip
Now we prove (3). Take $j_0^{k}=0$ and  write the Birkhoff sum $\sum_{j = 0}^{n_k^2-1}\varphi(f^{j}(x))$ by
\begin{align}
\sum_{j = 0}^{n_k^2-1}\varphi(f^{j}(x))
	& = \sum_{l = 1}^{r_k^{}}\sum_{j = 0}^{n_k^1-1}\varphi(f^{j+(l-1)n_k^1 +\sum_{t\leq l-1}\;j_t^{k}}~(x)) \nonumber \\
	& + \sum_{l = 1}^{s_k^{}}\sum_{j = 0}^{\widetilde{n}_k^{}-1}\varphi(f^{j+ (l-1)\widetilde{n}_k^{} +r_k^{} n_k^1 +\sum_{t\leq r_k^{} + l-1}\;j_t^{k} }~(x))  \nonumber \\
	& + \sum_{i=1}^{r_k^{}+s_k^{}-1} \sum_{j=0}^{j_i^{k} -1} \varphi(f^{j+\chi_i+\sum_{t\leq i-1} j_t^{k}}(x)),
	\label{eq:3rd}
\end{align}
where
\begin{equation*}
\chi_i = \left\{
\begin{array}{ll}
i n_k^1~~~~~~~~ ~ & ~\hbox {~ if }~ 1\leq i \leq r_k^{} \\
r_k n_k^1+(i-r_k^{})\widetilde{n}_k^{} & ~\hbox {~ if }~ r_k^{} < i \leq r_k^{}+s_k^{}-1.
\end{array}
\right.
\end{equation*}
The third expression in the right hand-side of \eqref{eq:3rd} satisfies
$$
\Big\| \sum_{i=1}^{r_k^{}+s_k^{}-1} \sum_{j=0}^{j_i^{} -1} \varphi(f^{j+\chi_i+\sum_{t\leq i-1}~j_t^{}}(x)) \Big\|
	\leq (r_k+s_k-1)(m(\vep)+1)\|\varphi\|_\infty.
$$
Using ~\eqref{skrk} one can estimate the Birkhoff sums in terms of the periods of shadowing and the remainder terms
as follows:
\begin{align*}
\Big\| \sum_{j = 0}^{n_k^2-1}\varphi(f^{j}(x)) & - n_k^2\, \int \varphi d\mu_2\Big\|
	\leq (r_k^{} n_k^1 + s_k^{} \widetilde{n}_k) \var(\varphi, \vep) \\
	&+ \Big\|\sum_{l = 1}^{r_k}\Big(\sum_{j = 0}^{n_k^1-1} \varphi(f^{j+(l-1)n_k^1
	+\sum_{t\leq l-1} j_t^{}} (x_l^k)) - n_k^1\int \varphi d\mu_1\Big)\Big\|\\	
	&+ \Big\|\sum_{l = 1}^{s_k^{}} \Big( \sum_{j = 0}^{\widetilde{n}_k^{}-1}\varphi(f^{j+r_k^{} n_k^1 +\sum_{t\leq l-1}j_{r_k^{} + t}
	+(l-1)\widetilde{n}_k^{}}~(\widetilde{x}_k^{})) -  \widetilde{n}_k^{} \int \varphi d\nu\Big)\Big\|\\
	& + \Big\| n_k^2 \int \varphi d\mu_2 - r_k^{} n_k^1 \int \varphi d\mu_1 - s_k \widetilde{n}_k^{} \int \varphi d\nu\Big\|\\
	& + (r_k+s_k-1)(m(\vep)+1)\|\varphi\|_\infty\\
	&\leq n_k^2 \, (\var(\varphi, \vep) +  \zeta_k^{})
	 + |t n_k^2 - r_k n_k^1| \| \varphi\|_\infty + |(1-t) n_k^2 - s_k \widetilde{n}_k^{} | \| \varphi\|_\infty \\
	& +(r_k+s_k-1)(m(\vep)+1)\|\varphi\|_\infty.
\end{align*}
Dividing all terms in the previous estimate by $n_k^{2}$
and using ~\eqref{zeronk2} - \eqref{skrk} we conclude that item (3) holds.

\medskip
We are now left to prove item (2).  First, computations similar to \eqref{eq:3rd} for the potential $\psi \in C^0(X,\mathbb R)$ yield
\begin{align}
\sum_{x\in \widehat{S_k^2}} \exp \{\sum_{i=0}^{n_k^2 -1} \psi(f^i (x))\}
	& \ge \big[\sum_{z\in  S_k^1} \exp \{\sum_{i=0}^{n_k^1 -1} \psi(f^i (z))\} \Big]^{r_k} \nonumber \\
	& \quad \times e^{-n_k^2 \big[  \var(\psi,\vep) + \frac{s_k \widetilde{n}_k}{n_k^2} |\psi|_\infty + \frac{r_k+s_k -1}{n_k^2}  |\psi|_\infty
	\big]} \nonumber \\
	& \ge  \exp\; (n_k^2 [ t P_{top}^{}(f,\psi,\vep) - \var(\psi,\vep) - 5\gamma ] ) \label{estimateMk2}
\end{align}
for every large $k\ge 1$.
Here we used equations ~\eqref{skrk}, \eqref{zeronk2} and Lemma~\ref{constnkesk}.
Recall the definition of $\widehat{S_k^2}$ and consider the shadowing point map
\[
\begin{array}{ccc}
\mathcal S : \{0,1, \dots, m(\vep)\}^{r_k+s_k-1} \times (S_k^1)^{r_k} \times \{\widetilde{x_k}\}^{s_k} & \to &  \widehat{S_k^2} \subset X \\
	(\underline j, \underline x, (\widetilde{x_k}, \dots, \widetilde{x_k})) & \mapsto &  y(\underline j, \underline x,
                            	 \widetilde{x_k},s_k).
\end{array}
\]
Observe that
$$
\widehat{S_k^2}
		=  \bigsqcup_{\underline j} S_k^2(\underline j)
		=  \bigsqcup_{\underline j} \text{Image}(\mathcal S(\underline j, \cdot))
$$
where the union is over all possible $\underline j \in \{0, 1, \dots, m(\vep)\}^{r_k}$. Now, equations ~\eqref{skrk} and  ~\eqref{estimateMk2}, the separability condition proved in item (1) and the pigeonhole principle ensure that there exists a string
$\underline{j_k^{}}= (j_1^{k}, \dots , j_{r_k^{}+s_k^{}-1}^{k})$  such that
\begin{align*}
\sum_{x\in S_k^2(\underline j_k)} \exp \{\sum_{i=0}^{n_k^2 -1} \psi(f^i (x))\}
	& \ge \frac{1}{(m(\vep)+1)^{r_k+s_k-1}}
		\sum_{x\in \widehat{S_k^2}} \exp \{\sum_{i=0}^{n_k^2 -1} \psi(f^i (x))\} \\
	& \ge \exp\; (n_k^2 [ t P_{top}^{}(f,\psi,\vep) - \var(\psi,\vep) - 6\gamma ] )
\end{align*}
for every large $k\ge 1$. The set $S_k^2= S_k^2 (\underline{j}_k)$ satisfies the requirements of item (2). This proves the lemma.
\end{proof}

\subsubsection{Construction of sets of points with oscillatory behavior}

Consider the sequences $\{S_k\}_k$ and $\{n_k\}_k$ given by
\begin{equation*}
S_k = \left\{
\begin{array}{rl}
S_k^1~,~ & \hbox { if } k ~\text{is odd} \\
S_k^2~,~ & \hbox { if } k ~\text{is even},
\end{array}
\right.
\quad\text{ and}\quad
n_k = \left\{
\begin{array}{rl}
n_k^1 ~,~ & \hbox { if } k ~\text{is odd} \\
n_k^2 ~,~ & \hbox { if } k ~\text{is even}.
\end{array}
\right.
\end{equation*}
Lemmas \ref{constnkesk} and \ref{lemaestimSK2} ensure that
\begin{equation}~\label{cardsk}
M_k := \sum_{x\in S_k} \exp \{\sum_{i=0}^{n_k -1} \psi(f^i (x))\} \geq \exp\; (n_k [ t P_{top}^{}(f,\psi,\vep) - \var(\psi,\vep) - 6\gamma ])
\end{equation}
for every large $k\geq 1$. Since we will construct sets of points that interpolate between those in the sets $S_k$
within a $\frac{\vep}{2^k}$-distance (in the Bowen metric) we need the transition times $m_k=m(\frac{\vep}{2^k})$
to be negligible in comparison with the total size of the orbits. For that, choose
a strictly increasing sequence of integers  $\{N_k^{}\}_{k\geq 0}$ so that $N_0^{} = 1$,
\begin{align}\label{limdosN}
& \lim_{k \to~\infty}\frac{n_{k+1}^{} + m_k}{N_{k}^{}}=0,~~~~~~ \quad\text{and}\quad ~~~~~~~~~\nonumber \\
& \lim_{k \to~\infty}\frac{1+N_1 (n_{1}^{}+m_k )+ \dots + N_k^{}  (n_{k}^{}+m_k )}{N_{k+1}^{}}=0.
\end{align}

For any fixed $k\ge 1$ and any string $\underline x =(x_1^k, x_2^k, \dots, x_{N_k}^k) \in  S_{k}^{N_k^{}}$ there exists a point
$y=y (\underline{x}) \in X$ which satisfies
\begin{eqnarray*}
d_{n_k^{}}(f^{a_j^{}}(y), x_{i_j^{}}^{k} )<\frac{\vep}{2^k},~~~~\forall~
j=1, 2, \dots , N_k^{}
\end{eqnarray*}

\noindent where
\begin{equation*}\label{defaj}
a_j =\begin{cases}
0 & \text{, ~~if $j=1$}\\
(j - 1)n_{k}^{} + \sum_{r=1}^{j - 1} p_{k,r}^{} & \text{,~~if $j= 2, \dots , N_k $}
\end{cases}
\end{equation*}
and $p_{k,r}^{}$ are the transition time functions, bounded by $m_k$.

Define
$$
C_{k} = \big\{y(\underline x ) \in X : \underline x = (x_1^k, x_2^k, \dots, x_{N_k}^k) \in  S_{k}^{N_k^{}}\big\}
$$
and
$
c_{k}^{}= N_{k}^{} n_{k}^{} + \sum_{r=1}^{N_k^{} - 1} p_{k,r}^{}
$~
(it is a function on $C_k$).
Proceeding as before, it is not hard to check that for any fixed $\underline s=(s_1, \dots, s_{N_k-1})$ (with all coordinates bounded by $m_k$) the subset  $C_k(\underline s) \subset C_k$ with these prescribed transition times is a $(3\vep,N_{k}^{} n_{k}^{} + \sum_{i=1}^{N_k^{} - 1} s_i)$-separated set.
Using ~\eqref{cardsk} and the pigeonhole principle, there exists
$\underline{s_k^{}}= (s_1^{k}, \dots , s_{N_k-1}^{k})$  so that the set
$$
{C}_{k}^{}(\underline s_k)
	=\big\{y(\underline{x})\in C_k^{} \colon \underline{x}\in S_k^{N_k} \;\text{and} \;
	p_{k, 1}^{}=s_1^{k}, ~ \dots ~,~p_{k, N_{k}-1}^{} = s_{N_k-1}^{k} \big\}
$$
satisfies
\begin{align}
\sum_{x\in {C}_{k}^{}(\underline s_k)} \exp \{\sum_{i=0}^{c_k^{} -1} \psi(f^i (x))\}
	& \ge \frac{1}{(m_k^{}+1)^{N_k^{}}}
		\sum_{x\in C_k} \exp \{\sum_{i=0}^{c_k -1} \psi(f^i (x))\} \nonumber  \\
	& \ge \exp\; (n_k N_k [ t P_{top}^{}(f,\psi,\vep) - \var(\psi,\vep) - 6\gamma ] )\nonumber  \\
	& \quad \times e^{-c_k \big[  \var(\psi,\frac\vep{2^k}) + \frac{N_k -1}{c_k}  |\psi|_\infty + \frac{N_k \log m_k}{c_k}
	\big]}   \nonumber \\
	& \ge \exp\; (c_k [ t P_{top}^{}(f,\psi,\vep) - \var(\psi,\vep) - 7\gamma ] ) \label{descardck}
\end{align}
for every large $k\ge 1$, where
$
c_k^{} = n_k^{} N_k^{} + \sum_{i=1}^{N_k-1} s_i^k
$
is constant for all points of the set ${C}_{k}^{}(\underline s_k)$.
We used that $ \log m_k^{}/ n_k^{} \to 0$ (cf.  \eqref{defdenk2}) and $(n_k N_k)/ c_k \to1$
as $k\to\infty$.
As before we will denote ${C}_{k}^{}(\underline s_k)$ simply by $C_k$.

\medskip
We now construct points whose averages
oscillate between $\int \varphi \, d\mu_1$ and $\int \varphi \, d\mu_2$. Define $T_1 = C_1$ and $t_1^{}=c_1^{}$, and we define the families $(T_k)_{k\ge 1}$ and $(t_k)_k$ recursively.
If $x \in T_k$ and $y \in C_{k+1}$ there exists a point $z := z(x, y) \in X$ and $0\le p_{k+1}\le m_{k+1}$ such that
\begin{equation*}
d_{t_k^{}}(x, z) < \frac{\vep}{2^{k+1}}  \quad\text{and}\quad d_{c_{k+1}^{}}(f^{t_k+p_{k+1}}(z), y)<\frac{\vep}{2^{k+1}}.
\end{equation*}

Define the set
\begin{align*}
T_{k+1}= \left\{ z=z(x, y)~\in X  ~: x\in~T_k,~y\in~C_{k+1} \right\}
\end{align*}
and  $t_{k+1}^{} = t_k^{} + p_{k+1}^{} + c_{k+1}^{}$ (it is a function on $T_k$).
Using the previous argument once more as above we conclude that there exists {$0\le \underline p_{k+1} \le m_{k+1}$}
such that $T_{k+1}(\underline p_{k+1}) \subset T_{k+1}$ is a $(2\vep, t_k^{} + \underline p_{k+1}+ c_{k+1})$-separated set.
We will keep denoting $T_{k+1}(\underline p_{k+1})$ by $T_{k+1}$ for notational simplicity.
In particular, if $z=z(x,y) \in T_{k+1}$ then
\begin{equation*}
d_{t_k^{}}(x, z) < \frac{\vep}{2^{k+1}} ~ \quad\text{and}\quad ~d_{c_{k+1}^{}} \big(f^{t_k+\underline p_{k+1} }(z), y\big)<\frac{\vep}{2^{k+1}}.
\end{equation*}

\subsubsection{Construction of a fractal set with large topological pressure}\label{subsec:entropy}

Define
\begin{eqnarray*}
F_k^{} = \bigcup_{z \in  {T}_k^{}} ~\overline{B_{t_k^{}}^{}\big(z, \frac{\vep}{2^{k}}\big)}
	\quad\text{and}\quad
	F = \bigcap_{k \geq 1} F_k^{}.
\end{eqnarray*}

The previous set $F$ depends on $\vep$, but we shall omit its dependence for notational simplicity.
As $F_{k+1}^{}\subset F_k^{}$ for all $k\ge 1$ then $F$ is the (non-empty) intersection of a sequence of compact and nested subsets.
In the present subsection we will prove the following:
\begin{equation}\label{eqboundF}
P_F(f,\psi, \vep) \ge  t P_{top}^{}(f,\psi,\vep) - \var(\psi,\vep) - 9\gamma .
\end{equation}

\begin{remark}\label{rmunicrepp}
Every point $x \in F$ can be uniquely represented by an itinerary $\underline{x} = (\underline{x_1}, \underline{x_2}, \underline{x_3},
\dots)$ where each $\underline{x_i}= (x_{1}^i, \dots , x_{N_k^{}}^{i}) \in {S}_i^{N_i}$. We will keep denoting by
$y(\underline{x}_i^{}) \in {C}_i$ the point in ${C}_i$ determined by the sequence $\underline x_i$ with a sequence
$\underline{s}_i=(s_1^{i}, \dots , s_{N_i-1}^{i})$ of transition times, and by
$z_{i}^{}(\underline{x}) = z(z_{i-1}^{}(\underline x), y(\underline{x}_{i}^{})) \in {T}_{i}$ the element constructed using the points
$z_{i-1}^{}(\underline x) \in T_{i-1}$ and  $y(\underline{x}_{i}^{}) \in C_i$, and with transition time $\underline p_i$.
\end{remark}

We will use the following pressure distribution principle:

\begin{proposition}\cite[Proposition 2.4]{TMP} \label{exiseq-s}
Let $f : X \to X$ a continuous map on a compact metric space $X$ and let $Z \subset X$ be a Borel set.
Suppose there are $\vep > 0$,  $s \in \mathbb R$, $K>0$ and a sequence of probability measures  $(\mu_k^{})_k$
satisfying:
\begin{itemize}
\item[(i)] $\mu_k \to \mu$ and $\mu (Z)> 0$, and
\item[(ii)] $ \limsup_{k \to \infty} \mu_k (B_{n}(x, \vep))\leq K \exp\{-n s+\sum_{i=0}^{n-1} \psi(f^i(x)\}$ for every large $n$ and every ball $B_n^{}(x, \vep)$ such that $B_{n}(x, \vep)\cap Z \neq \emptyset$.
\end{itemize}
Then, $P_{Z}^{}(\psi, \vep) \geq s$.
\end{proposition}

Assume first that $\htop(f)<\infty$ (hence $P_{top}(f,\psi)<\infty$, by the variational principle). We use the previous proposition to estimate $P_F(f,\psi,\vep)$. Consider a sequence $(\mu_k)_k$ of measures on $F$ as follows: take
$\nu_k^{} = \sum_{z\in {T}_k }  \Psi(z)\, \delta_z^{}$ and its normalization

\begin{eqnarray*}
\mu_k^{} = \frac{1}{Z_k} \nu_k^{}
	\quad \text{where}\quad
	Z_k=\sum_{z\in T_k} \Psi(z),
\end{eqnarray*}
and for every $z = z(\underline{x}_1,  \dots ,  \underline{x}_k) \in T_k$ and  $\underline{x}_i =(x_1^{i}, \dots , \underline{x}_{N_i}^{i})\in  S_i^{N_i}$ we set
$$
\Psi(z) = \prod_{i=1}^{k} \, \prod_{l=1}^{N_i}\, \exp S_{n_i}\, \psi(x_{l}^{i}).
$$
We will prove that $(\mu_k^{})_k^{}$ satisfies the hypothesis of Proposition \ref{exiseq-s}. Given $n\geq 1$, let
$B=B_{n}^{}(q, \vep / 2)$ be a dynamical ball that intersects $F$, let $k\ge 1$ be such that $t_k^{} \leq n < t_{k+1}^{}$, and let
$0 \le j \le N_{k+1}^{}-1$ be so that
\begin{equation}\label{choiceN}
t_k^{} + j n_{k+1}^{} + \sum_{1\le i \le j} s_i^{k+1}
	\leq n
	< t_k^{} + (j+1) n_{k+1}^{} + \sum_{1\le i \le j+1} s_i^{k+1}.
\end{equation}

\begin{lemma}\label{lem-muB}
If $\mu_{k+1}(B) >0$ then
$$
\nu_{k+1}(B) \le e^{S_n\psi(q) + n  \var (\psi,\vep)  + ( \sum_{i=1}^{k} N_i m_i + j\,m_{k+1})  |\psi|_\infty} \, M_{k+1}^{N_{k+1}-j}.
$$
\end{lemma}

\begin{proof}
If $\mu_{k+1}(B)>0$ then ${T}_{k+1} \cap B \neq \emptyset$. Let $z=z(x, y)\in {T}_{k+1} \cap B$ determined by $x\in {T}_k$
and $y=y(\underline x_1^{}, \dots, \underline x_{N_{k+1}^{}}^{})\in {C}_{k+1}$ and let $\underline p_{k+1}$ be so that
\begin{equation*}
d_{t_k^{}}(z, x) < \frac{\vep}{2^{k+1}}  \quad\text{and}\quad d_{c_{k+1}^{}}(f^{t_k^{} +  \underline p_{k+1}}(z), y)<\frac{\vep}{2^{k+1}}.
\end{equation*}

Since $z\in B_{n}^{}(q, \vep / 2)$ and $n\ge t_k$ then $d_{t_k}(x, q) \le d_{t_k}(x, z) + d_{t_k}(z, q) <\vep$.
Using the definition of $n$ and the fact that $d_{n}(z, q)<\frac{\vep}{2}$ we have that
$$
d_{n_{k+1}^{}}^{} ( f^{t_k^{} + (l-1) n_{k+1}^{} + \sum_{0\le i \le l-1} s_i^{k+1} }(z),f^{t_k^{} + (l-1) n_{k+1}^{} + \sum_{0\le i \le l-1} s_i^{k+1} }(q))
<\frac{\vep}{2}
$$
for all $l = 1, \dots , j.$
Moreover, by construction $d_{c_{k+1}^{}}(f^{t_k^{} + \underline p_{k+1}}(z), y)<\frac{\vep}{2^{k+1}}$.  This implies on the following
estimates for blocks of size $n_{k+1}$:
$$
d_{n_{k+1}^{}}^{}( f^{t_k^{} + \underline p_{k+1} + (l-1)\;n_{k+1}^{} + \sum_{0\le i \le l-1} s_i^k }(z),
	f^{ (l-1)\;n_{k+1}^{} + \sum_{0\le i \le l-1} s_i^k }(y))<\frac{\vep}{2^{k+1}}
$$
for all $l = 1, \dots , N_{k+1}$.  Using that $y=y(\underline x_1^{}, \dots, \underline x_{N_{k+1}^{}}^{})\in {C}_{k+1}$ we also have
$$
d_{n_{k+1}^{}}^{}( f^{ (l-1)\;n_{k+1}^{} + \sum_{0\le i \le l-1} s_i^k }(y),~x_{l^{}}^{k+1})<\frac{\vep}{2^{k+1}}
$$
for all $l = 1, \dots , j.$
Altogether the previous estimates imply
\begin{equation}\label{eq:imply-sep}
d_{n_{k+1}^{}}^{}( f^{t_k^{} + \underline p_{k+1} + (l-1)\;n_{k+1}^{} + \sum_{0\le i \le l-1} s_i^k }(z),
~x_{l}^{k+1} )<2\;\vep
\end{equation}
for all $l = 1, \dots , j.$

We remark that
if $\hat z=z(\hat x, \hat y) \in T_{k+1} \cap B$ then $d_{t_k}(\hat x, q)<\vep$ and, consequently, $d_{t_k}(\hat x, x)<2\vep$.
 Since ${T}_k$ is $(t_k, 2\vep)$ separated and $n \geq t_k$  then $x = \hat x$. Moreover, the previous estimates also ensure (cf. \eqref{eq:imply-sep})
that
$
d_{n_{k+1}^{}}^{}(x_{l}^{k+1}, {\hat x}_{l}^{k+1} )<4\;\vep
$
for all $l = 1, \dots, j$.
However, as $x_{i^{}}^{k+1}$ and ${\hat x}_{i^{}}^{k+1} $ belong to $S_{k+1}$, which is a $(n_{k+1}^{}, 4\vep)$-separated set then $x_{i^{}}^{k+1}= \hat{x}_{i^{}}^{k+1}$ for every $i=1, \dots, j$.

The previous argument implies that all elements $z=z(x,y) \in T_{k+1} \cap B$ with
$x\in T_k$ and $y=(\underline x_1, \dots, \underline x_{N_{k+1}}) \in C_{k+1}$ may only differ in the last
$N_{k+1}-j$ elements of $S_{k+1}$. Therefore, by the choice of $k$ and $j$ in \eqref{choiceN},
\begin{align*}
\nu_{k+1}(B)
	& = \sum_{z\in {T}_{k+1} \cap B}  \Psi(z) \nonumber \\
	&\le \Psi(x) \big[ \prod_{l=1}^{j} \exp S_{n_{k+1}} \psi (x_{l}^{k+1}) \big]
		\, \sum_{l= j+1}^{N_{k+1}} \exp(S_{n_{k+1}} \psi (x_{l}^{k+1})) \nonumber  \\
	&= \Psi(x) \big[ \prod_{l=1}^{j} \exp S_{n_{k+1}} \psi (x_{l}^{k+1}) \big]
		\prod_{l=j+1}^{N_{k+1}} \sum_{ \tilde x \in S_{k+1}} \exp(S_{n_{k+1}} \psi (\tilde x )) \nonumber  \\
	&= \Psi(x) \big[ \prod_{l=1}^{j} \exp S_{n_{k+1}} \psi (x_{l}^{k+1}) \big] \, M_{k+1}^{N_{k+1}-j}  \label{eq:upbnu} \\
	& \le e^{S_n\psi(q) + n  \var (\psi,\vep)  + ( \sum_{i=1}^{k} N_i m_i + j\,m_{k+1})  |\psi|_\infty} \, M_{k+1}^{N_{k+1}-j}
\end{align*}
which proves the lemma.
\end{proof}

\begin{lemma}\label{lemadesigwk}
$
Z_k \,(M_{k+1})^j \geq \exp (n (t\,P_{top}(f, \psi,\vep) - \var(\psi,\vep)- 8\gamma))
$
for all $k \gg 1$.
\end{lemma}

\begin{proof}
By the variational principle and the fact that $\psi$ is bounded away from zero and infinity assumption (i) is equivalent
to $P_{top}(f,\psi)<\infty$. A simple computation shows that $Z_k={M_k}^{N_k}$ for every $k\ge 1$. Moreover, using
\begin{align*}
n  & < t_k^{} + (j+1) (n_{k+1}^{} + m_{k+1}) \\
	& = \sum_{i=1}^k n_i^{} N_i^{}+ \sum_{i=1}^k \big(\underline p_i + \sum_{l=1}^{N_i-1} s^i_l \big) + (j+1) (n_{k+1}^{} + m_{k+1}) \\
	& \le \sum_{i=1}^k [(n_i^{} +m_i) N_i^{}  + m_i]  + (j+1) (n_{k+1}^{} + m_{k+1})
\end{align*}
equation ~\eqref{cardsk}, and that $m_i \ll n_i \ll N_i$ for every $1\leq i \leq k$ we get
\begin{align*}
Z_k\,M_{k+1}^j &= M_{1}^{N_1} \dots M_{k}^{N_k} M_{k+1}^{j}\\
	&\geq  \exp\; \big(\, (\sum_{i=1}^k N_i n_i  + j n_{k+1}) [ t P_{top}^{}(f,\psi,\vep) - \var(\psi,\vep) - 6\gamma ]\,\big) \\
	& \geq  \exp (t(\sum_{i=1}^{k} N_i( n_i+ m_i) + j(n_{k+1}+m_{k+1}))
	[ t P_{top}^{}(f,\psi,\vep) - \var(\psi,\vep) - 7\gamma ] )\\
	& \geq  \exp (n [ t\, P_{top}^{}(f,\psi,\vep) - \var(\psi,\vep) - 8\gamma ])
\end{align*}
for all large $k$, proving the lemma.
\end{proof}

\begin{corollary}\label{lem-muB2}
The following holds:
$$
\limsup_{k \to \infty} \mu_k^{} (B_{n}^{}(q, \vep/2))\leq
\exp (- n (t\,P_{top}(f, \psi,\vep) - \var(\psi,\vep)- 9\gamma) + S_n\psi(q)).
$$
\end{corollary}

\begin{proof} By Lemmas \ref{lem-muB} and \ref{lemadesigwk} we get
\begin{align*}
\mu_{k+1}^{}(B)
	& \le \frac{1}{Z_k \,M_{k+1}^{N_{k+1}}}
	e^{S_n\psi(q) + n  \var (\psi,\vep)  + ( \sum_{i=1}^{k} N_i m_i + j\,m_{k+1})  |\psi|_\infty} \, M_{k+1}^{N_{k+1}-j} \\
	& = \frac{1}{Z_k \,M_{k+1}^{j}}
	e^{S_n\psi(q) + n  \var (\psi,\vep)  + ( \sum_{i=1}^{k} N_i m_i + j\,m_{k+1})  |\psi|_\infty} \\
	& \le \exp (- n (t\,P_{top}(f, \psi,\vep) - \var(\psi,\vep)- 9\gamma) + S_n\psi(q))
\end{align*}
for all large $k$,
proving the corollary.
\end{proof}

Now, an argument similar e.g. to \cite[p.1200]{CKL} ensures that any accumulation point $\mu$ of $\mu_k^{}$ satisfies $\mu(F)=1$. Since the hypothesis of Proposition~\ref{exiseq-s}
are satisfied we conclude that
$
P_F(f,\psi,\vep) \geq t\,P_{top}(f, \psi,\vep) - \var(\psi,\vep)- 9\gamma
$
proving equation \eqref{eqboundF}.

\medskip

Finally, by the variational principle for the topological entropy, in the case that
$\sup_{\mu\in \mathcal M_1(f)} h_\mu(f)=\htop(f)=+\infty$ (hence $P_{top}(f,\psi)=+\infty$) the argument follows with minor modifications. Indeed,
one can repeat the previous arguments and prove that
for any $K>0$ and $t\in (0,1)$ there exist invariant probability measures $\mu_1,\mu_2$ so that $\mu_1$ is ergodic,
$h_{\mu_1}(f) +\int \psi \,d\mu_1>K$, $h_{\mu_2}(f) +\int \psi \,d\mu_2> tK$
and $\int \varphi \, d\mu_1 \neq \int \varphi \, d\mu_2$.
The same argument as before shows that for any given $\vep,\gamma>0$ there exists a fractal set $F\subset X_{\varphi,f}$ such that
$$
P_{X_{\varphi,f}}(f,\psi,\vep) \ge P_F(f,\psi, \vep) \ge t K-\var(\psi,\vep)-9\gamma,
$$
leading to the conclusion that $P_{X_{\varphi,f}}(f,\psi) \ge K$. Since $K>0$ is arbitrary and $\psi$ is bounded above and below
then
$$P_{X_{\varphi,f}}(f,\psi)=P_{top}(f,\psi) = h_{X_{\varphi,f}}(f) = \htop(f)=+\infty$$
as claimed.

\subsubsection{$F$ is formed by points with historic behavior}

In order to complete the proof of Theorem~\ref{thm:B} it suffices to prove that ${F} \subset X_{\varphi,f}$.

\begin{proposition}\label{finirreg}~~~ $F\subset X_{\varphi, f}$.
\end{proposition}

\begin{proof}

 Let $x \in F$, and set $\chi(k)^{} =1$ if $k$ odd, and $\chi(k)^{} = 2$ otherwise. By Remark~\ref{rmunicrepp} let $y_k^{} := y(\underline{x_k}) \in C_k$ and $z_k^{} = z_k^{}(\underline{x})\in T_k.$
First we prove that points in $C_k$ have time averages close to $\int \varphi d\mu_{\chi(k)}^{}$.
More precisely, we claim that
\begin{eqnarray}\label{exp1-FinI}
\Big\| \frac{1}{c_k^{}} \sum_{j = 0}^{c_k^{}-1}\varphi(f^{j}(y_k^{})) - \int \varphi d\mu_{\chi(k)}^{}\Big\|  ~\to 0~~~\text{as}~~k \to \infty.
\end{eqnarray}

Recalling that $c_k^{}=N_{k}^{} n_{k}^{} + \sum_{i=1}^{N_k^{} - 1} s_i^k$ and
$0\le s_i^k\le m_k^{}$
for every $i$, one can write
\begin{align*}
\Big\|\sum_{j = 0}^{c_k^{}-1} \varphi(f^{j}(y_k^{})) & - c_k^{}\int \varphi d\mu_{\chi(k)}^{}\Big\| \\
	 & \leq \Big\|\sum_{j = 1}^{N_k^{}} \sum_{i = 0}^{n_k^{}-1} \varphi(f^{i+ (j-1)n_k + \sum_{i=1}^{j - 1} s_i^k}(y_k^{})) - n_k^{} N_k\int \varphi d\mu_{\chi(k)}^{}\Big\| \\
	& + 2 m_k^{} (N_k^{} -1) \|\varphi\|_\infty\\
	&\leq \sum_{j = 1}^{N_k^{}}\sum_{i = 0}^{n_k^{}-1} \| \varphi(f^{i+ (j-1)n_k + \sum_{i=1}^{j - 1} s_i^k}(y_k^{})) -  \varphi(f^{i}(x_{j^{}}^{k})) \| \\
	&+ \Big\| \sum_{j = 1}^{N_k^{}} \big[\sum_{j = 0}^{n_k^{}-1} \varphi(f^{i}(x_{j^{}}^{k}))  - n_k^{}\;\int \varphi \, d\mu_{\chi(k)}^{}\big] \Big\| 	
	\\
	& + 2 m_k^{} (N_k^{} -1) \|\varphi\|_\infty\\
&\leq N_k^{} n_k^{}\Big(\var(\varphi, \frac{\vep}{2^k}) + \zeta_k^{} \Big) + 2 m_k^{} (N_k^{} -1) \|\varphi\|_\infty.
\end{align*}

Using that $\lim_{k\to \infty} \frac{n_k^{} N_k^{}}{c_k^{}}=1$ and $\lim_{k\to \infty} \frac{m_k^{} N_k^{}}{c_k^{}}=0$ we conclude that
\begin{eqnarray*}
\Big\|\frac{1}{c_k^{}}\sum_{j = 0}^{c_k^{}-1} \varphi(f^{j}(y_k^{})) - \int \varphi d\mu_{\chi(k)}^{}\Big\|
	\leq \frac{N_k^{} n_k^{}}{c_k^{}}\Big(\var(\varphi, \frac{\vep}{2^k}) +\zeta_k^{}\Big)
	+ \frac{2 m_k^{} (N_k^{} -1) }{c_k^{}} \|\varphi\|_\infty
\end{eqnarray*}
tends to zero as $k\to\infty$, which proves the claim.

\medskip
Now, take any point $x\in F$. By definition for every $k\ge 1$ there exists $z_k=z(z_{k-1},y_k)\in T_k$ so that $d_{t_k^{}}^{}(x, z_k^{})\le \frac{\vep}{2^{k}}$. Using that $t_k=c_k+ \underline p_k+t_{k-1}$ and triangular inequality we get
$$
d_{c_k^{}}^{}(f^{t_k^{}-c_k^{}}(x), y_k^{})
	\leq d_{t_k^{}}^{}(f^{t_k^{}-c_k^{}}(x), f^{t_k^{}-c_k^{}}(z_k^{}))
	+d_{c_k^{}}^{}(f^{t_k^{}-c_k^{}}(z_k^{}), y_k^{})
	< \frac{\vep}{2^{k-1}}.
$$
In particular, 
\begin{align*}
\Big\|\frac{1}{c_k^{}}\sum_{j = 0}^{c_k^{}-1} & \varphi(f^{t_k^{}-c_k^{}+j}(p))  - \int \varphi d\mu_{\chi(k)}^{}\Big\|\\
	& \leq \var\Big(\varphi, \frac{\vep}{2^{k-1}}\Big)
	+ \Big\|\frac{1}{c_k^{}}\sum_{j = 0}^{c_k^{}-1} \varphi(f^{j}(y_k^{})) - \int \varphi d\mu_{\chi(k)}^{}\Big\|
\end{align*}
tends to zero as $k\to\infty$.
Using that $\lim_{k \to \infty} \frac{c_k^{}}{t_k^{}}=1$ and dividing the $t_k$-time average in their first $t_k-c_k$ summands and
the second $c_k$ summands, a simple computation shows
\begin{align*}
\Big\|\frac{1}{t_k^{}}\sum_{j = 0}^{t_k^{}-1} \varphi(f^{j}(x)) & - \frac{1}{c_k^{}}\sum_{j = 0}^{c_k^{}-1} \varphi(f^{j+t_k^{}-c_k^{}}(x)) \Big\|
	\leq2\frac{t_k^{} - c_k^{}}{t_k^{}} \|\varphi\|_\infty \to 0
\end{align*}
as $k\to\infty$. Altogether we get that
$
\lim_{k\to \infty}~\big\|\frac{1}{t_k^{}}\sum_{j = 0}^{t_k^{}-1} \varphi(f^{j}(p)) - \int \varphi d\mu_{\chi(k)}^{} \big\| = 0,
$
which proves the proposition.
\end{proof}

\subsubsection{Proof of Theorem \ref{thm:B}}

We will consider the case of topological entropy and metric mean dimension, as the argument that proves that the historic set carries full topological pressure is completely analogous. We note that
$$
0\le \underline{mdim}(f) \le \overline{mdim}(f) \le \htop(f) \le +\infty
$$
and that $\overline{mdim}(f)=0$ whenever $\htop(f) < +\infty$. For that reason we distinguish the following cases:

\medskip
\noindent \emph{Case 1:} $0=mdim(f) < \htop(f) < +\infty$.
\smallskip

It is immediate that $\text{mdim}_{X_{\varphi, f}}^{} (f) = \text{mdim} (f)=0$. It remains to prove that
$h_{X_{\varphi, f}}(f)=\htop(f)$.  
Let $\gamma,\vep>0$ be arbitrary and small and let $t\in (0,1)$ be close to one, given by Lemma~\ref{exist-medmu}
(when $\psi\equiv 0$).
Corollary \ref{lem-muB2} ensures that $(\mu_k)_k$ satisfies the hypothesis
of Proposition~\ref{exiseq-s} with $s= t\;h_{top}^{}(f, 2\vep) - 9\gamma$ and $K=1$. Since $F\subset X_{\varphi,f}$ then
Proposition~\ref{exiseq-s}  ensures that
$$
h_{X_{\varphi,f}}(f,\vep)
	\ge h_{F}(f, \vep)
	\geq t\; h_{top}^{}(f, 2\vep) - 9\gamma.
$$
Using that $\gamma>0$ is arbitrary and that $t\to 1$ as $\gamma\to 0$  we conclude that
\begin{equation}\label{eq:finalent}
h_{X_{\varphi,f}}(f,\vep)
	\geq h_{top}^{}(f, 2\vep).
\end{equation}
Taking the limit as $\vep\to 0$ we get the desired equality  $h_{X_{\varphi,f}}(f) = h_{top}^{}(f)$.

\medskip
\noindent \emph{Case 2:} $0\le \underline{mdim}(f) \le \overline{mdim}(f) < \htop(f) = +\infty$.
\smallskip

The argument which ensures that $h_{X_{\varphi,f}}(f) = +\infty$ was explained at the end of Subsection~\ref{subsec:entropy}.
We are left to prove that
$\underline{\text{mdim}}_{X_{\varphi, f}}^{} (f) = \underline{\text{mdim}} (f)$ and
$\overline{\text{mdim}}_{X_{\varphi, f}}^{} (f) = \overline{\text{mdim}} (f)$.
This is now immediate because inequality \eqref{eq:finalent} guarantees that
$$
\frac{ h_{top}^{}(f, \vep) }{-\log\vep}
	\ge 
	\frac{h_{X_{\varphi,f}}(f, \vep)}{-\log\vep}
	\geq \frac{ h_{top}^{}(f, 2\vep) }{-\log\vep}
$$
for all 
$\vep>0$.
 This proves the theorem.
\subsection{Proof of Proposition~\ref{p:A}}

Let $X$ be a compact Riemannian manifold. Theorem~1 in \cite{PPSS} ensures that there exists 
$C^0$-Baire residual subset $\mathfrak{R} \subset \text{Homeo}(X)$ such that every $f\in \mathfrak{R}$
has infinitely many periodic points of some finite period (actually such  periodic points are uncountable, 
cf. pp 246 in \cite{PPSS}).

Fix $f\in \mathfrak{R}$ and let $n=n(f)\ge 1$ be such that the set of periodic points $Per_n(f)$ is infinite.
Choose a sequence $(p_i)_{i\ge 1}$ of points in $Per_n(f)$ that generate pairwise disjoint periodic orbits.
Then $E_f:=\{ \varphi \in C^0(X,\mathbb R^d)  \colon \mu \mapsto \int \varphi\, d\mu \; \text{is constant} \}$
is contained in the countable intersection
\begin{align*}
\bigcap_{\substack{ (i,j) \in \mathbb N  \times \mathbb N \\ i \neq j}}
		\Big\{ \varphi \in C^0(X,\mathbb R^d) \colon  \frac1n \sum_{s=0}^{n-1} \varphi (f^s(p_i))
		= \frac1n \sum_{s=0}^{n-1} \varphi (f^s(p_j)) \Big\}
\end{align*}
of $C^0$-closed sets with empty interior. 
The set $\mathfrak{R}_f= C^0(X,\mathbb R^d) \setminus E_f$ satisfies the requirements of the proposition.

\subsection{Proof of Corollary~\ref{cor:A}}

The proof of the corollary relies on the genericity of the gluing orbit property on isolated chain recurrent classes of
 the non-wandering set.
Let $\widetilde{\mathcal{R}_0}$ be as in the proof of Corollary \ref{lem1} and let $\mathfrak{R} \subset \text{Homeo}(X)$
and $\mathfrak{R}_f \subset C^0(X,\mathbb R^d)$, $f\in \mathfrak{R}$, be given by Proposition~\ref{p:A}. 
Notice that $\widetilde{\mathcal{R}_0} \cap \mathfrak{R}\subset \text{Homeo}(X)$ 
and 
$$
\widehat{\mathfrak R}:=\bigcup_{f\in \widetilde{\mathcal R_0} \cap \mathfrak R} 
		\{f\} \times \mathfrak{R}_f \subset \text{Homeo}(X) \times C^0(X,\mathbb R^d)
$$
are $C^0$-Baire generic subsets.

Fix $f\in \widetilde{\mathcal{R}_0} \cap \mathfrak{R}$. 
If $\Gamma\subset CR(f)$ is a isolated chain recurrent class then $f\mid_\Gamma$ satisfies the gluing orbit property
(cf. Corollary \ref{lem1}). Hence, Theorem~\ref{thm:B} (applied to the map $f\mid_\Gamma$) implies that 
$\htop(f\mid_{\Gamma\cap X_{\varphi,f}})=\htop(f\mid_\Gamma)$ for any $\varphi\in C^0(X,\mathbb R^d)$ such that 
$\Gamma\cap X_{\varphi,f} \neq \emptyset$. This proves item (1) is satisfied 
by pairs $(f,\varphi)\in \widehat{\mathfrak R}$.

Now, take $f\in \widetilde{\mathcal{R}_0} \cap \mathfrak{R}$ and assume that $CR(f)=X$.
Since $f\in \widetilde{\mathcal{R}_0}$ then the set of periodic points is dense in $CR(f)=\Omega(f)=X$. 
Hence, using Corollary \ref{lem1} once more, $f$ satisfies the gluing orbit property. 
Moreover,  $X_{\varphi,f}\neq \emptyset$ 
for every $(f,\varphi) \in \widetilde{\mathfrak R}$ (by Lemma~\ref{obsat} and Proposition~\ref{p:A}).
Item (2) in the corollary follows also as a consequence of Theorem~\ref{thm:B}. 
This completes the proof of the corollary.

\section{Some comments and further questions}\label{sec:questions}

To finish we will make some comments on related concepts and future perspectives. First, the general concept of multifractal analysis is to decompose the phase space  in subsets of points which have a similar dynamical behavior and to describe the size of each of such subsets from the geometrical or  topological viewpoint. We refer the reader to the introduction of \cite{Olsen} and references therein for a great historical account. The study of the topological pressure or Hausdorff dimension of the level and the irregular sets can be traced back to Besicovitch. Such a multifractal analysis program has been carried out successfully to deal with self-similar measures and Birkhoff averages~\cite{Olsen,OlsenWinter,PW97,ZC}, among other applications. We expect our methods to be applied in other related problems as the multifractal analysis of level sets
for Birkhoff averages.

A different question that can be endorsed concerns the concept of localized entropy. In \cite{14KW}, studied the directional $H(v)$
entropy (in the direction of a rotation vector $v$) introduced in \cite{Jenkinson} (we refer the reader to \cite{Jenkinson, 14KW}
for the definition).  They prove that,
if the localized entropy satisfies some mild continuity assumptions, the localized entropy
associated to locally maximal invariant set  of $C^{1+\alpha}$-diffeomorphisms is entirely determined by the exponential
growth rate of periodic orbits whose rotation vectors are sufficiently close to $v$
(cf. \cite[Theorem 5]{14KW} for the precise statement).
While it is not hard to check that any fixed rotation vector $v$ there exist points whose pointwise rotation
set coincides with $v$ in the case of maps with the gluing orbit property, we expect that the inequality
$H(v) \le \lim_{\vep\to 0} \limsup_{n\to\infty} \frac1n \log \# Per(v,n,\vep)$
holds.

One different question concerns the Hopf ratio ergodic theorem. More precisely, although we did not pursue this here, it is
most likely that our results can describe the set of points with historic behavior for quotients of Birkhoff sums in the spirit of \cite{BLV,TMP}, with possible applications to the case of suspension flows over continuous maps with the gluing orbit property, considered in \cite{BV}.

Theorems~\ref{thm:D} and ~\ref{thm:E}, dealing with isolated chain recurrent classes, 
can be thought as a first step in the understanding of rotation sets for $C^0$-generic
homeomorphisms isotopic to the identity on the torus. While the dynamics of topologically generic homeomorphisms
is rather complex \cite{AHK}, the general picture still remains out of reach. A natural question which could contribute to the
understanding of the global picture is wether all chain recurrent classes of generic homeomorphisms satisfy the gluing orbit
property. 

Finally, the convexity of the rotation set played a key role on the rotation theory for homeomorphisms on the $2$-torus. Hence, we expect Theorem~\ref{thm:E} to contribute for the development of the rotation theory for generic conservative homeomorphisms on tori. In particular, taking into account \cite{P14}, an interesting open question is wether the rotation set of a $C^0$-generic homeomorphisms on $\mathbb T^d$ homotopic to the identity is a rational polyhedron. This has been announced recently in 
\cite{BLV}.

\subsection*{Acknowledgements} This work is part of the PhD thesis of the first author at UFBA and it was partially supported by
CNPq-Brazil and CAPES-Brazil. It is a pleasure to thank  A. Koropecki and F. Tal for many useful comments,
and to {P.~Oprocha} for a discussion that led us to consider isolated chain recurrent classes in Theorem~\ref{thm:D}.

\end{document}